\documentclass[a4paper,12pt]{amsart}
\usepackage{amsmath}
\usepackage{amsthm}
\usepackage{amssymb}
\usepackage{amscd}
\usepackage{mathrsfs}
\usepackage[all]{xy}
\usepackage{enumerate}
\usepackage{graphics}

\numberwithin{equation}{section}

\theoremstyle{plain}
\newtheorem{thm}{Theorem}[section]
\newtheorem{dfn}[thm]{Definition}
\newtheorem{prop}[thm]{Proposition}
\newtheorem{cor}[thm]{Corollary}
\newtheorem{lem}[thm]{Lemma}
\newtheorem*{thm*}{Theorem}

\theoremstyle{definition}
\newtheorem{conj}[thm]{Conjecture}
\newtheorem{rem}[thm]{Remark}
\newtheorem{exa}[thm]{Example}

\newcommand{\nc}{\newcommand}
\nc{\Prop}{\begin{prop}}
\nc{\enprop}{\end{prop}}

\def\dim{\mathop{\mathrm{dim}}\nolimits}

\def\mod{\mathop{\mathrm{mod}}\nolimits}

\def\Hom{\mathop{\mathrm{Hom}}\nolimits}

\def\Ad{\mathop{\mathrm{Ad}}\nolimits}
\def\ord{\mathop{\mathrm{ord}}\nolimits}
\def\id{\mathop{\mathrm{id}}\nolimits}

\newcommand{\mf}[1]{{\mathfrak{#1}}}
\newcommand{\mb}[1][m]{{\mathfrak{#1}}}

\newcommand{\lr}[2]{{\langle #1,#2 \rangle}}

\nc{\on}{\operatorname}
\newcommand{\C}{{\mathbb C}}
\newcommand{\Q}{\mathbb {Q}}

\newcommand{\Z}{{\mathbb Z}}
\newcommand{\B}{{\mathcal{B}}}

\newcommand{\gl}{{\mathfrak{gl}}}

\newcommand{\R}{{\mathbb{R}}}

\newcommand{\seteq}{\mathbin{:=}}

\newcommand{\m}{\mathfrak{m}}

\newcommand{\Lemma}{\begin{lem}}
\newcommand{\enlemma}{\end{lem}}

\newcommand{\g}{{\mathfrak{g}}}

\newcommand{\isoto}[1][]%
{{\mathop{\buildrel{\sim}\over\longrightarrow}\limits_{#1}}}
\renewcommand{\hom}{\operatorname{\it \mathscr{H}\kern-.25em om}}

\newcommand{\M}{{\mathcal{M}}}

\newcommand{\eq}{\begin{eqnarray}}
\newcommand{\eneq}{\end{eqnarray}}

\newcommand{\eqn}{\begin{eqnarray*}}
\newcommand{\eneqn}{\end{eqnarray*}}

\newcommand{\QED}{\end{proof}}
\newcommand{\Proof}{\begin{proof}}

\newcommand{\soplus}{\mathop{\mbox{\normalsize$\bigoplus$}}\limits}

\newcommand{\cl}{\colon}

\newcommand{\ba}{\begin{array}}
\newcommand{\ea}{\end{array}}

\newcommand{\epi}{\twoheadrightarrow}

\newcommand{\set}[2]{\left\{#1 \mid #2 \right\}}

\newcommand{\hs}{\hspace*}

\newcommand{\eqsub}{\begin{subequations}\begin{eqnarray}}
\newcommand{\eneqsub}{\end{eqnarray}\end{subequations}}

\newcommand{\ol}{\overline}

\newcommand{\A}{\mathbf{A}}

\renewcommand{\le}{\leqslant}
\renewcommand{\ge}{\geqslant}

\nc{\la}{\lambda}
\nc{\lam}{\lambda}
\nc{\U}[1][\g]{U_q(#1)}
\nc{\te}{\tilde{e}}
\nc{\tei}{\tilde{e}_i}
\nc{\tf}{\tilde{f}}
\nc{\tfi}{\tilde{f}_i}
\nc{\tU}{\widetilde U_q(\g)}
\nc{\tE}{\widetilde{E}}
\nc{\tF}{\widetilde{F}}

\nc{\BZ}{{\mathbb{Z}}}
\nc{\al}{\alpha}
\nc{\qs}{{q}}
\nc{\lan}{\langle}
\nc{\ran}{\rangle}
\nc{\re}{{\mathrm{re}}}
\nc{\wt}{\operatorname{wt}}
\nc{\Uf}[1][\g]{U^-_q(#1)}
\nc{\Ue}{U^+_q(\g)}
\nc{\eps}{\varepsilon}
\nc{\vphi}{\varphi}
\nc{\sphi}{\varphi^*}
\nc{\seps}{\varepsilon^*}

\nc{\nn}{\nonumber}
\def\max{{\mathop{\mathrm{max}}}}
\nc{\vp}{\varpi}

\nc{\cls}{{\operatorname{cl}}}

\nc{\Wt}{{\operatorname{Wt}}}
\nc{\Us}{U'_q(\g)}
\nc{\La}{\Lambda}
\nc{\ro}{{\rm(}}
\nc{\rf}{{\rm)}}
\nc{\norm}{{\mathrm{norm}}}
\nc{\qbox}{\quad\mbox}
\nc{\braid}{{\mathfrak{B}}}
\nc{\dt}[1]{\tilde{\tilde #1}}
\nc{\Sn}{S^{{\mathrm{norm}}}}
\nc{\aff}{{\mathrm{aff}}}
\nc{\rk}{{\mathrm{rk}}}
\nc{\tQ}{\widetilde{Q}}
\nc{\tP}{\widetilde{P}_\theta}
\nc{\tW}{\widetilde{W}}
\nc{\Dyn}{\mathrm{Dyn}}
\nc{\tD}{\widetilde{\Delta}}
\nc{\height}{{\operatorname{ht}}}
\nc{\bl}{\bigl}
\nc{\br}{\bigr}
\nc{\Hecke}{\mathrm{H}}
\nc{\HA}{\Hecke^{\mathrm{A}}}
\nc{\HB}{\Hecke^{\mathrm{B}}}
\nc{\K}{\mathbf{K}}
\newcommand{\scbul}{{\,\raise1pt\hbox{$\scriptscriptstyle\bullet$}\,}}
\nc{\vac}{{\phi}}
\nc{\Bt}[1][\g]{\B_\theta(#1)}
\nc{\Btg}{\Bt[\gl_\infty]}
\nc{\be}{\begin{enumerate}}
\nc{\ee}{\end{enumerate}}
\nc{\low}{{\mathrm{low}}}
\nc{\upper}{{\mathrm{up}}}
\nc{\lw}{{\mathrm{low}}}
\nc{\Zodd}{\Z_{\mathrm{odd}}}
\nc{\Ft}[1][n]{\mathbb{P}\mathrm{ol}_{#1}}
\nc{\Ftf}[1][n]{\widetilde{\mathbb{P}\mathrm{ol}}_{#1}}
\nc{\KA}{\on{K}^{\mathrm{A}}}
\nc{\KB}{\on{K}^{\mathrm{B}}}
\nc{\Fc}[1][{n,m}]{\mathbf{F}_{#1}}
\nc{\tphi}{\tilde{\varphi}}
\nc{\CO}{\mathscr{O}}
\nc{\pbw}[1]{\lan #1 \ran}
\nc{\dv}[1]{{[#1]}}
\nc{\Mt}{{\M}_\theta}
\nc{\tVt}{\widetilde{V}_\theta(0)}
\nc{\Vt}[1][0]{V_\theta(#1)}
\nc{\Lt}[1][0]{L_\theta(#1)}
\nc{\tvac}{\widetilde{\vac}}
\nc{\ssum}{\mathop{\mbox{\normalsize{${\sum}$}}}\limits}
\nc{\bnum}{\be[{\rm(i)}]}
\nc{\enum}{\ee}
\nc{\Pt}{P_\theta}
\nc{\suml}{\sum\limits}
\nc{\tL}{\widetilde L}
\nc{\ltcr}{\underset{\mathrm{cry}}{<}}
\nc{\lecr}{\underset{\mathrm{cry}}{\le}}
\nc{\tp}{\tilde{\rho}}
\nc{\dpo}{{(\mspace{-3mu}(}}
\nc{\dpf}{{)\mspace{-3mu})}}
\nc{\odd}{\mathrm{odd}}
\nc{\CB}{\mathrm{B}}
\nc{\Bz}{\CB_\theta(0)}
\nc{\Bs}{\CB_\theta(\la)}
\nc{\Gl}{\mathrm{G}}
\nc{\Gth}{\Gl_\theta}

\begin{document}
\title{Symmetric Crystals for $\gl_\infty$}
\author{Naoya Enomoto}
\address{Research Institute for Mathematical Sciences\\
Kyoto University\\
Kyoto 606--8502, Japan
}
\email{henon@kurims.kyoto-u.ac.jp}

\author{Masaki Kashiwara}
\email{masaki@kurims.kyoto-u.ac.jp}
\thanks{The second author is partially supported by 
Grant-in-Aid for Scientific Research (B) 18340007,
Japan Society for the Promotion of Science.}

\date{\today}

\begin{abstract}
In the preceding paper,
we formulated a conjecture 
on the relations between certain classes of irreducible representations
of affine Hecke algebras of type B
and symmetric crystals for $\gl_\infty$.
In the present paper, we prove the existence of the
symmetric crystal
and the global basis for $\gl_\infty$.
\end{abstract}


\maketitle
\tableofcontents
\section{Introduction}
Lascoux-Leclerc-Thibon (\cite{LLT})
conjectured the relations between the representations of 
Hecke algebras of {\em type A} and the crystal bases of the affine Lie algebras
of type A. Then, S.~Ariki (\cite{A}) observed that it should be understood 
in the setting of affine Hecke algebras
and proved the LLT conjecture in a more general framework.
Recently, we
presented the notion of symmetric crystals and
conjectured that 
certain classes of irreducible representations 
of the affine Hecke algebras of {\em type B} are described by
symmetric crystals for $\gl_\infty$ (\cite{EK}).

The purpose of the present paper is to prove
the existence of symmetric crystals in the case of $\gl_\infty$.

Let us recall the Lascoux-Leclerc-Thibon-Ariki theory.
Let $\HA_n$ be the affine Hecke algebra of type A of degree $n$.
Let $\KA_n$ be the Grothendieck group of
the abelian category of finite-dimensional $\HA_n$-modules,
and $\KA=\oplus_{n\ge0}\KA_n$.
Then it has a structure of Hopf algebra by the restriction and
the induction.
The set $I=\C^*$ may be regarded as a Dynkin diagram
with $I$ as the set of vertices and with edges between $a\in I$ and $ap_1^2$.
Here $p_1$ is the parameter of 
the affine Hecke algebra usually denoted by $q$.
Let $\g_I$ be the associated Lie algebra, and
$\g_I^-$ the unipotent Lie subalgebra. Let $U_I$ be the group associated
to $\g_I^-$.
Hence $\g_I$ is isomorphic to a direct sum of copies of $A^{(1)}_\ell$
if $p_1^2$ is a primitive $\ell$-th root of unity
and to a direct sum of copies of $\gl_\infty$ if $p_1$ has an infinite order.
Then $\C\otimes \KA$ is isomorphic to the algebra $\CO(U_I)$
of regular functions on $U_I$.
Let $\U[\g_I]$ be the associated quantized enveloping algebra.
Then $\Uf[\g_I]$ has an upper global basis 
$\{\Gl^{\upper}(b)\}_{b\in \CB(\infty)}$.
By specializing $\soplus \C[q,q^{-1}]\Gl^\upper(b)$
at $q=1$, we obtain $\CO(U_I)$.
Then the LLTA-theory says that the
elements associated to irreducible $\HA$-modules
corresponds to the image of the upper global basis.

In \cite{EK},
we gave analogous conjectures for affine Hecke algebras of type B.
In the type B case, we have to replace 
$\Uf[\g_I]$ and its upper global basis with
symmetric crystals (see \S\;\ref{subsec:symcry}).
It is roughly stated as follows.
Let $\HB_n$ be the affine Hecke algebra of type B of degree $n$.
Let $\KB_n$ be the Grothendieck group of
the abelian category of finite-dimensional modules over $\HB_n$,
and $\KB=\oplus_{n\ge0}\KB_n$.
Then $\KB$ has a structure of a Hopf bimodule over $\KA$.
The group $U_I$ has the anti-involution $\theta$ induced by
the involution $a\mapsto a^{-1}$ of $I=\C^*$. Let $U_I^\theta$
be the $\theta$-fixed point set of $U_I$.
Then $\CO(U_I^\theta)$ is a quotient ring of $\CO(U_I)$.
The action of $\CO(U_I)\simeq\C\otimes\KA$ on $\C\otimes\KB$,
in fact, descends to the action of $\CO(U_I^\theta)$.

We introduce 
$V_\theta(\la)$ (see \S\;\ref{subsec:symcry}),
a kind of the $q$-analogue of $\CO(U_I^\theta)$.
The conjecture in \cite{EK} is then:
\bnum
\item
$V_\theta(\la)$ has a crystal basis and a global basis.
\item
$\KB$ is isomorphic to a specialization of $V_\theta(\la)$ at $q=1$
as an $\CO(U_I)$-module, and
the irreducible representations correspond
to the upper global basis of $V_\theta(\la)$ at $q=1$.
\enum

\noindent
{\bf Remark.}\quad
In \cite{KM}, Miemietz and the second author gave an analogous conjecture
for the affine Hecke algebras of type D.

\medskip
In the present paper, we prove that
$V_\theta(\la)$ has a crystal basis and a global basis for
$\g=\gl_\infty$ and $\la=0$.

More precisely, let $I=\Z_{\odd}$ be the set of odd integers.
Let $\al_i$ ($i\in I$) be the simple roots with
$$(\al_i,\al_j)=\begin{cases}
2&\text{if $i=j$,}\\
-1&\text{if $i=j\pm2$,}\\
0&\text{otherwise.}
\end{cases}$$
Let $\theta$ be the involution of $I$ given by $\theta(i)=-i$.
Let $\Bt[\gl_\infty]$ be the algebra over $\K\seteq\Q(q)$
generated by $E_i$, $F_i$, and
invertible elements $T_i$ \ro$i\in I$\rf\ 
satisfying the following defining relations:
\begin{enumerate}[{\rm(i)}]
\item the $T_i$'s commute with each other,
\item
$T_{\theta(i)}=T_i$ for any $i$,
\item
$T_iE_jT_i^{-1}=q^{(\al_i+\al_{\theta(i)},\al_j)}E_j$ and
$T_iF_jT_i^{-1}=q^{(\al_i+\al_{\theta(i)},-\al_j)}F_j$
for $i,j\in I$,
\item
$E_iF_j=q^{-(\al_i,\al_j)}F_jE_i+
(\delta_{i,j}+\delta_{\theta(i),j}T_i)$
for $i,j\in I$,
\item
the $E_i$'s and the $F_i$'s satisfy the Serre relations
(see Definition~\ref{U_q(g)} (4)).
\end{enumerate}
Then there exists a unique irreducible $\Bt[\gl_\infty]$-module
$\Vt[0]$ with a generator $\vac$ satisfying
$E_i\vac=0$ and $T_i\vac=\vac$ (Proposition~\ref{prop:Vtheta}).
We define the endomorphisms
$\tE_i$ and $\tF_i$ of $\Vt$ by
\[
\tE_ia=\sum_{n \ge 1}F_i^{(n-1)}a_n, \quad \tF_ia=\sum_{n \ge 0}f_i^{(n+1)}a_n,
\] 
when writing 
\[
a=\sum_{n \ge 0}F_i^{(n)}a_n \quad \text{with} \ E_ia_n=0.
\]
Here $F_i^{(n)}=F_i^n/[n]!$ is the divided power.
Let $\A_0$ be the ring of functions $a\in \K$ which
do not have a pole at $q=0$.
Let $\Lt$ be the $\A_0$-submodule of $\Vt$
generated by the elements $\tF_{i_1}\cdots\tF_{i_\ell}\vac$
($\ell\ge0$, $i_1,\ldots,i_\ell\in I$).
Let $\Bz$ be the subset of $\Lt/q\Lt$ consisting of 
the $\tF_{i_1}\cdots\tF_{i_\ell}\vac$'s.
In this paper, we prove the following theorem.
\begin{thm*}[Theorem~\ref{main:cr}]
\bnum
\item
$\tF_iL_\theta(0)\subset L_\theta(0)$
and $\tE_iL_\theta(0)\subset L_\theta(0)$,
\item
$\Bz$ is a basis of $L_\theta(0)/\qs L_\theta(0)$,
\item
$\tF_i\Bz\subset\Bz$,
and
$\tE_i\Bz\subset \Bz\sqcup\{0\}$,
\item
$\tF_i\tE_i(b)=b$ for any $b\in \Bz$ such that $\tE_ib\not=0$,
and $\tE_i\tF_i(b)=b$ for any $b\in  \Bz$.
\enum
\end{thm*}

Let $-$ be the bar operator of $\Vt$.
Namely, $-$ is a unique endomorphism of $\Vt$ such that
$\ol{\vac}=\vac$, $\ol{av}=\bar{a}\bar{v}$ and $\ol{F_iv}=F_i\bar{v}$ for
$a\in \K$ and $v\in \Vt$. Here $\bar{a}(q)=a(q^{-1})$.

Then we prove the existence of global basis:
\begin{thm*}[Theorem~\ref{th:gl}]
Let $\Vt_\A$ be the smallest submodule of $\Vt$ over 
$\A\seteq\Q[q,q^{-1}]$ such that it contains $\vac$ and is stable
by the $F_i^{(n)}$'s.
\bnum
\item
For any $b\in \Bz$, there exists a unique
$\Gth^\lw(b)\in\Vt_\A\cap\Lt$ such that
$\ol{\Gth^\lw(b)}=\Gth^\lw(b)$ and $b=\Gth^\lw(b)\ \mod q \Lt$,
\item
$\Lt=\soplus_{b\in \Bz}\A_0\Gth^\lw(b)$,
$\Vt_\A=\soplus_{b\in \Bz}\A \Gth^\lw(b)$ and
$\Vt=\soplus_{b\in \Bz}\K \Gth^\lw(b)$.
\enum
\end{thm*}
We call $ \Gth^\lw(b)$ the {\em lower global basis}.
The $\Bt[\gl_\infty]$-module $\Vt$ has a unique symmetric bilinear form
$(\scbul,\scbul)$ such that
$(\vac,\vac)=1$ and $E_i$ and $F_i$ are transpose to each other.
The dual basis to  $\{\Gth^\lw(b)\}_{b\in\Bz}$
with respect to $(\scbul,\scbul)$ 
is called an {\em upper global basis}.

\bigskip
Let us explain the strategy of our proof of these theorems.
We first construct a PBW type basis $\{\Pt(\mb)\vac\}_{\mb}$ of $\Vt$ 
parametrized by the $\theta$-restricted multisegments $\mb$.
Then, we explicitly calculate the actions of $E_i$ and $F_i$
in terms of the PBW basis $\{\Pt(\mb)\vac\}_{\mb}$.
Then, we prove that the PBW basis gives a crystal basis by the
estimation of the coefficients of these actions.
For this we use a criterion for crystal bases
(Theorem~\ref{th:crcr}).

\section{General definitions and conjectures}
\subsection{Quantized universal enveloping algebras and its reduced $q$-analogues}
We shall recall the quantized universal enveloping algebra
$U_q(\g)$.
Let $I$ be an index set (for simple roots),
and $Q$ the free $\Z$-module with a basis $\{\al_i\}_{i\in I}$.
Let $(\scbul,\scbul)\cl Q\times Q\to\Z$ be
a symmetric bilinear form such that
$(\al_i,\al_i)/2\in\Z_{>0}$ for any $i$ and
$(\al_i^\vee,\al_j)\in\Z_{\le0}$ for $i\not=j$ where
$\al_i^\vee\seteq2\al_i/(\al_i,\al_i)$.
Let $q$ be an indeterminate and set 
$\K\seteq\Q(\qs)$.
We define its subrings $\A_0$, $\A_\infty$ and $\A$ as follows.
\begin{eqnarray*}
\A_0&=&\set{f\in\K}{\text{$f$ is regular at $q=0$}},\\
\A_\infty&=&\set{f\in\K}%
{\text{$f$ is regular at $q=\infty$}},\\
\A&=&\Q[\qs,\qs^{-1}].
\end{eqnarray*}
\begin{dfn}\label{U_q(g)}
The quantized universal enveloping algebra $U_q(\g)$ is the $\K$-algebra
generated by elements $e_i,f_i$ and invertible
elements $t_i\ (i\in I)$
with the following defining relations.
\begin{enumerate}[{\rm(1)}]

\item The $t_i$'s commute with each other.

\item
$t_je_i\,t_j^{-1}=q^{(\al_j,\al_i)}\,e_i\ $
and $\ t_jf_it_j^{-1}=q^{-(\al_j,\al_i)}f_i\ $
for any $i,j\in I$.

\item\label{even} $\lbrack e_i,f_j\rbrack
=\delta_{ij}\dfrac{t_i-t_i^{-1}}{q_i-q_i^{-1}}$
for $i$, $j\in I$. Here $q_i\seteq q^{(\al_i,\al_i)/2}$.

\item {\rm(}{\em Serre relation}{\rm)} For $i\not= j$,
\begin{eqnarray*}
\sum^b_{k=0}(-1)^ke^{(k)}_ie_je^{(b-k)}_i=0, \ 
\sum^b_{k=0}(-1)^kf^{(k)}_i
f_jf_i^{(b-k)}=0.
\end{eqnarray*}
Here $b=1-(\al_i^\vee,\al_j)$ and
\begin{eqnarray*}
\ba{l}
e^{(k)}_i=e^k_i/\lbrack k\rbrack_i!\,,\; f^{(k)}_i=f^k_i/\lbrack k
\rbrack_i!\ , \ \lbrack k\rbrack_i=(q^k_i-q^{-k}_i)/(q_i-q^{-1}_i)\,, \ 
\lbrack k\rbrack_i!=\lbrack 1\rbrack_i\cdots \lbrack k\rbrack_i\,.
\ea
\end{eqnarray*}
\end{enumerate}
\end{dfn}
Let us denote by $\Uf$ (resp.\ $\Ue$)
the $\K$-subalgebra of $\U$ generated by the $f_i$'s (resp.\ the $e_i$'s).

Let $e'_i$ and $e^*_i$ be the operators on $\Uf$ defined by
$$[e_i,a]=\dfrac{(e^*_ia)t_i-t_i^{-1}e'_ia}{q_i-q_i^{-1}}\quad(a\in\Uf).$$
These operators satisfy the following formulas similar to derivations:
\eq
&&\ba{l}
e_i'(ab)=e_i'(a)b+(\Ad(t_i)a)e_i'b,\quad\\[2ex]
e_i^*(ab)=ae_i^*b+(e_i^*a)(\Ad(t_i)b).
\ea\label{eq:der}
\eneq
The algebra $\Uf$ has a unique symmetric bilinear form $(\scbul,\scbul)$
such that $(1,1)=1$ and
\[
(e'_ia,b)=(a,f_ib)\quad\text{for any $a,b\in\Uf$.}
\]
It is non-degenerate and satisfies $(e^*_ia,b)=(a,bf_i)$.
The left multiplication of $f_j,e'_i$ and $e_i^*$ have the commutation relations
\[
e'_if_j=q^{-(\al_i,\al_j)}f_je'_i+\delta_{ij}, \ 
e_i^*f_j=f_je_i^*+\delta_{ij}\Ad(t_i),
\]
and both the $e_i'$'s and the $e^*_i$'s satisfy the Serre relations.

\begin{dfn}
The reduced $q$-analogue $\B(\mf{g})$ of $\mf{g}$ is the 
$\K$-algebra generated by $e_i'$ and $f_i$. 
\end{dfn}

\subsection{Review on crystal bases and global bases}
Since $e_i'$ and $f_i$ satisfy the $q$-boson relation, 
any element $a \in U_q^{-}(\mf{g})$ can be uniquely written as
\[
a=\sum_{n \ge 0}f_i^{(n)}a_n \quad \text{with} \ e_i'a_n=0.
\]
Here $f_i^{(n)}=\dfrac{f_i^n}{[n]_i!}$. 
\begin{dfn}
We define the modified root operators $\widetilde{e}_i$ and $\widetilde{f}_i$ on $U_q^{-}(\mf{g})$ by 
\[
\widetilde{e}_ia=\sum_{n \ge 1}f_i^{(n-1)}a_n, \quad \widetilde{f}_ia=\sum_{n \ge 0}f_i^{(n+1)}a_n.
\] 
\end{dfn}
\begin{thm}[\cite{K}]
We define
\begin{eqnarray*}
L(\infty)&=&\sum_{\ell \ge 0,\,i_1, \ldots ,i_{\ell} \in I}
\A_0\tf_{i_1} \cdots \tf_{i_\ell} \cdot 1
 \subset U_q^{-}(\mf{g}), \\
\CB(\infty)&=&\set{\tf_{i_1} \cdots \tf_{i_\ell} \cdot 1\; \mod qL(\infty)}
{\ell \ge 0,i_1, \cdots ,i_{\ell} \in I} \subset L(\infty)/qL(\infty).
\end{eqnarray*}
Then we have
\bnum
\item $\widetilde{e}_iL(\infty) \subset L(\infty)$ 
and $\widetilde{f}_{i}L(\infty) \subset L(\infty)$, 
\item $\CB(\infty)$ is a basis of $L(\infty)/qL(\infty)$,
\item
$\widetilde{f}_i\CB(\infty) \subset \CB(\infty)$ 
and $\widetilde{e}_i\CB(\infty) \subset \CB(\infty) \cup \{0\}$. 
\enum
We call $(L(\infty),\CB(\infty))$ the {\em crystal basis} of $U_q^{-}(\mf{g})$.
\end{thm}
Let $-$ be the automorphism of $\K$ sending $\qs$ to $\qs^{-1}$.
Then $\ol{\A_0}$ coincides with $\A_\infty$. 

Let $V$ be a vector space over $\K$,
$L_0$ an $\A_0$-submodule of $V$,
$L_\infty$ an $\A_\infty$- submodule, and
$V_\A$ an $\A$-submodule.
Set $E:=L_0\cap L_\infty\cap V_\A$.

\begin{dfn}[\cite{K}]
We say that $(L_0,L_\infty,V_\A)$ is {\em balanced}
if each of $L_0$, $L_\infty$ and $V_\A$
generates $V$ as a $\K$-vector space,
and if one of the following equivalent conditions is satisfied.
\bnum
\item
$E \to L_0/\qs L_0$ is an isomorphism,
\item
$E \to L_\infty/\qs^{-1}L_\infty$ is an isomorphism,
\item
$(L_0\cap V_\A)\oplus
(\qs^{-1} L_\infty \cap V_\A) \to V_\A$
      is an isomorphism,
\item
$\A_0\otimes_\Q E \to L_0$, $\A_\infty\otimes_\Q E \to L_\infty$,
        $\A\otimes_\Q E \to V_\A$ and $\K \otimes_\Q E \to V$
are isomorphisms.
\enum
\end{dfn}

Let $-$ be the ring automorphism of $\U$ sending
$\qs$, $t_i$, $e_i$, $f_i$ to $\qs^{-1}$, $t_i^{-1}$, $e_i$, $f_i$.

Let $\U_\A$ be the $\A$-subalgebra of
$\U$ generated by $e_i^{(n)}$, $f_i^{(n)}$
and $t_i$.
Similarly we define
$\Uf_\A$.

\begin{thm}
$(L(\infty),L(\infty)^-,\Uf_\A)$ is balanced.
\end{thm}
Let 
\[
\Gl^{\text{low}}\colon L(\infty)/\qs L(\infty)\isoto 
E\seteq L(\infty)\cap L(\infty)^-
\cap \Uf_\A
\] 
be the inverse of $E\isoto L(\infty)/\qs L(\infty)$.
Then $\set{\Gl^{\text{low}}(b)}{b\in \mathbb{B}(\infty)}$ forms a basis of $\Uf$.
We call it a (lower) {\em global basis}.
It is first introduced by G. Lusztig (\cite{L})
under the name of ``canonical basis'' for the A, D, E cases.

\begin{dfn}
Let
\[
\set{\Gl^\upper(b)}{b \in \CB(\infty)}
\]
be the dual basis of $\set{\Gl^\low(b)}{b \in \CB(\infty)}$ 
with respect to the inner product $( \scbul,\scbul)$. 
We call it the upper global basis of $U_q^{-}(\mf{g})$.
\end{dfn}

\subsection{Symmetric crystals}
\label{subsec:symcry}

Let $\theta$ be an automorphism of
$I$ such that $\theta^2=\id$ and 
$(\al_{\theta(i)},\al_{\theta(j)})=(\al_i,\al_j)$.
Hence it extends to an automorphism of the root lattice $Q$
by $\theta(\al_i)=\al_{\theta(i)}$,
and induces an automorphism of $\U$.

\begin{dfn}\label{def:Bt}
Let $\B_\theta(\g)$ be the $\K$-algebra
generated by $E_i$, $F_i$, and
invertible elements $T_i$ \ro$i\in I$\rf\ 
satisfying the following defining relations:
\begin{enumerate}[{\rm(i)}]
\item the $T_i$'s commute with each other,
\item
$T_{\theta(i)}=T_i$ for any $i$,
\item
$T_iE_jT_i^{-1}=q^{(\al_i+\al_{\theta(i)},\al_j)}E_j$ and
$T_iF_jT_i^{-1}=q^{(\al_i+\al_{\theta(i)},-\al_j)}F_j$
for $i,j\in I$,
\item
$E_iF_j=q^{-(\al_i,\al_j)}F_jE_i+
(\delta_{i,j}+\delta_{\theta(i),j}T_i)$
for $i,j\in I$,
\item
the $E_i$'s and the $F_i$'s satisfy the Serre relations 
{\rm (Definition~\ref{U_q(g)} (4))}.
\end{enumerate}
\end{dfn}

We set $E_i^{(n)}=E_i^n/[n]_i!$
and $F_i^{(n)}=F_i^n/[n]_i!$.

\Lemma Identifying $\Uf$ with the subalgebra of $\Bt$ generated by the $F_i$'s,
we have
\eq
&&T_ia=\bl(\Ad(t_it_{\theta(i)})a\br)T_i,\\
&&E_ia=\bl(\Ad(t_i)a\br)E_i+e'_ia+\bl(\Ad(t_i)(e_{\theta(i)}^*a)\br)T_i
\label{eq:Ea}
\eneq
for $a\in\Uf$.
\enlemma
\Proof
The first relation is obvious.
In order to prove the second,
it is enough to show that if $a$ satisfies \eqref{eq:Ea}, then
$f_ja$ satisfies \eqref{eq:Ea}.
We have
\eqn
E_i(f_ja)&=&(q^{-(\al_i,\al_j)}f_jE_i+\delta_{i,j}+\delta_{\theta(i),j}T_i)a\\
&=&
q^{-(\al_i,\al_j)}f_j(\bl(\Ad(t_i)a\br)E_i+e'_ia+
\bl(\Ad(t_i)(e_{\theta(i)}^*a)\br)T_i)\\
&&\hs{25ex}
+\delta_{i,j}a+\delta_{\theta(i),j}\bl(\Ad(t_it_{\theta(i)})a\br)T_i\\
&=&(\bl(\Ad(t_i)(f_ja)\br)E_i+e'_i(f_ja)+
\bl(\Ad(t_i)(e_{\theta(i)}^*(f_ja)\br)T_i.
\eneqn
\QED
The following lemma can be proved in a standard manner and we omit the proof.
\Lemma
Let $\K[T_i^{\pm};i\in I]$ be the commutative $\K$-algebra
generated by invertible elements $T_i$ $(i\in I)$ with the defining relation
$T_{\theta(i)}=T_i$.
Then the map $\Uf\otimes\K[T_i^{\pm};i\in I]\otimes\Ue\to\Bt$ 
induced by the multiplication is bijective.
\enlemma

Let $\la\in P_+\seteq\set{\la\in \Hom(Q,\Q)}{
\text{$\lan \al_i^\vee,\la\ran\in\Z_{\ge0}$ for any $i\in I$}}$
be a dominant integral weight such that
$\theta(\la)=\la$.
\begin{prop}\label{prop:Vtheta}
\begin{enumerate}[{\rm(i)}]
\item
There exists a $\B_\theta(\g)$-module $V_\theta(\la)$
generated by a non-zero vector $\vac_\la$ such that
\be[{\rm(a)}]
\item
$E_i\vac_\la=0$ for any $i\in I$,
\item
$T_i\vac_\la=q^{(\al_i,\la)}\vac_\la$ for any $i\in I$,
\item
$\set{u\in V_\theta(\la)}{\text{$E_iu=0$ for any $i\in I$}}
=\K\vac_\la$.
\ee
Moreover such a $V_\theta(\la)$ is irreducible and
unique up to an isomorphism.
\item
there exists a unique symmetric bilinear form $(\scbul,\scbul)$
on $V_\theta(\la)$ such that $(\vac_\la,\vac_\la)=1$ and
$(E_iu,v)=(u,F_iv)$ for any $i\in I$ and $u,v\in V_\theta(\la)$,
and it is non-degenerate.
\end{enumerate}
\end{prop}

\begin{rem}\label{rem:tweight}
Set $P_\theta=\set{\mu\in P}{\theta(\mu)=\mu}$.
Then $\Vt[\la]$ has a weight decomposition
$$\Vt[\la]=\soplus_{\mu\in P_\theta}\Vt[\la]_\mu,$$
where
$\Vt[\la]_\mu=\set{u\in \Vt[\la]}{T_iu=q^{(\al_i,\mu)}u}$.
We say that an element $u$ of $\Vt[\la]$ has a $\theta$-weight
$\mu$ and write $\wt_\theta(u)=\mu$
if $u\in\Vt[\la]_\mu$.
We have $\wt_\theta(E_iu)=\wt_\theta(u)+(\al_i+\al_{\theta(i)})$
and $\wt_\theta(F_iu)=\wt_\theta(u)-(\al_i+\al_{\theta(i)})$.
\end{rem}

In order to prove Proposition~\ref{prop:Vtheta}, 
we shall construct two $\Bt$-modules.

\Lemma
Let $\Uf\vac'_\la$ be a free $\Uf$-module with a generator
$\vac'_\la$.
Then the following action
gives a structure of a $\Bt$-module on $\Uf\vac'_\la$ 
{\rm:} 
\eq&&
\left\{
\ba{rcl}
T_i(a\vac'_\la)&=&q^{(\al_i,\la)}(\Ad(t_it_{\theta(i)})a)\vac'_\la,\\[2pt]
E_i(a\vac'_\la)&=&
\bigl(e'_ia+q^{(\alpha_i,\la)}\Ad(t_i)(e^*_{\theta(i)}a)\bigr)\vac'_\la,\\[2pt]
F_i(a\vac'_\la)&=&(f_ia)\vac'_\la,
\ea\right.
\quad\text{for any $i\in I$ and $a\in\Uf$.}
\eneq
Moreover $\Bt/\ssum_{i\in I}
(\Bt E_i+\Bt(T_i-q^{(\al_i,\la)}))\to\Uf\vac'_\la$
is an isomorphism.
\enlemma
\Proof
We can easily check the defining relations in Definition~\ref{def:Bt}
except the Serre relations for the $E_i$'s.
For $i\not=j\in I$, set $S=\sum_{n=0}^{b}(-1)^nE_i^{(n)}E_jE_i^{(b-n)}$
where $b=1-\pbw{h_i,\alpha_j}$.
It is enough to show that the action of $S$ on $\Uf\vac'_\la$
is equal to $0$.
We can check easily that $SF_k=q^{-(b\alpha_i+\alpha_j,\alpha_k)}F_kS$.
Since $S\vac'_\la=0$, we have $S\Uf\vac'_\la=0$.

Hence $\Uf\vac'_\la$ has a $\Bt$-module structure.

The last statement is obvious.
\QED
\Lemma
Let $\Uf\vac''_\la$ be a free $\Uf$-module with a generator
$\vac''_\la$.
Then the following action
gives a structure of a $\Bt$-module on $\Uf\vac''_\la$ 
{\rm:} 
\eq&&\left\{\ba{rcl}
T_i(a\vac''_\la)&=&q^{(\al_i,\la)}(\Ad(t_it_{\theta(i)})a)\vac''_\la,\\[2pt]
E_i(a\vac''_\la)&=&(e'_ia)\vac''_\la,\\[2pt]
F_i(a\vac''_\la)&=&
\bigl(f_ia+q^{(\alpha_i,\la)}(\Ad(t_i)a)f_{\theta(i)}\bigr)\vac''_\la,
\ea\right.\quad \text{for any $i\in I$ and $a\in\Uf$.}
\label{eq:u''}
\eneq
Moreover, there exists a non-degenerate bilinear form
$\pbw{\scbul,\scbul}\cl \Uf\vac'_\la\times\Uf\vac''_\la\to\K$
such that
$\pbw{F_iu,v}=\pbw{u,E_iv}$,
$\pbw{E_iu,v}=\pbw{u,F_iv}$,
$\pbw{T_iu,v}=\pbw{u,T_iv}$ for $u\in \Uf\vac'_\la$
and $v\in \Uf\vac''_\la$, and
$\pbw{\vac'_\la,\vac''_\la}=1$.
\enlemma
\Proof
There exists a unique symmetric bilinear form $(\scbul,\scbul)$
on $\Uf$ such that $(1,1)=1$ and
$f_i$ and $e'_i$ are transpose to each other.
Let us define $\pbw{\scbul,\scbul}\cl \Uf\vac'_\la\times\Uf\vac''_\la\to \K$
by $\pbw{a\vac'_\la,b\vac''_\la}=(a,b)$ for $a\in \Uf$ and
$b\in\Uf$.
Then we can easily check
$\pbw{F_iu,v}=\pbw{u,E_iv}$, $\pbw{T_iu,v}=\pbw{u,T_iv}$.
Since $e^*_i$ is transpose to the right multiplication
of $f_i$, we have
$\pbw{E_iu,v}=\pbw{u,F_iv}$.
Hence the action of $E_i$, $F_i$, $T_i$ on $\Uf\vac''_\la$
satisfy the defining relations in Definition~\ref{def:Bt}.
\QED

\Proof[Proof of Proposition~\ref{prop:Vtheta}]
Since $E_i\vac''_\la=0$
and $\vac''_\la$ has a $\theta$-weight $\la$,
there exists a unique $\Bt$-linear morphism $\psi\cl
\Uf\vac'_\la\to\Uf\vac''_\la$
sending $\vac'_\la$ to $\vac''_\la$.
Let $V_\theta(\la)$ be its image $\psi(\Uf\vac'_\la)$.

(i)(c) follows from $\set{u\in\Uf}{\text{$e'_iu=0$ for any $i$}}=\K$
applying to $\Uf\vac''_\la\supset V_\theta(\la)$.
The other properties (a), (b) are obvious.
Let us show that $\Vt[\la]$ is irreducible.
Let $S$ be a non-zero $\Bt$-submodule.
Then $S$ contains a non-zero vector $v$ such that $E_iv=0$ for any $i$.
Then (c) implies that $v$ is a constant multiple of $\vac_\la$.
Hence $S=\Vt[\la]$.

Let us prove (ii).
For $u,u'\in\Uf\vac'_\la$, set $\dpo u,u'\dpf=\pbw{u,\psi(u')}$.
Then it is a bilinear form on $\Uf\vac'_\la$ which
satisfies 
\eq
&&\dpo \vac'_\la,\vac'_\la\dpf=1,\ 
\dpo F_iu,u'\dpf=\dpo u,E_iu'\dpf,\ 
\dpo E_iu,u'\dpf=\dpo u,F_iu'\dpf,\ 
\dpo T_iu,u'\dpf=\dpo u,T_iu'\dpf.\label{eq:dp}
\eneq
It is easy to see that a bilinear form which satisfies \eqref{eq:dp}
is unique.
Since $\dpo u',u\dpf$ also satisfies \eqref{eq:dp},
$\dpo u,u'\dpf$ is a symmetric bilinear form on $\Uf\vac'_\la$.
Since $\psi(u')=0$ implies
$\dpo u,u'\dpf=0$, 
$\dpo u,u'\dpf$ induces a symmetric bilinear form on $\Vt[\la]$.
%
Since $(\scbul,\scbul)$ is non-degenerate on $\Uf$,
$\dpo\scbul,\scbul\dpf$ is a non-degenerate symmetric bilinear
form on $V_\theta(\la)$.

\QED

\Lemma
There exists a unique endomorphism $-$ of $\Vt[\la]$
such that $\ol{\vac_\la}=\vac_\la$ and
$\ol{av}=\bar{a}\bar{v}$, $\ol{F_iv}=F_i\bar{v}$ for any $a\in \K$
and $v\in\Vt[\la]$.
\enlemma
\Proof
The uniqueness is obvious.

Let $\xi$ be an anti-involution of $\Uf$ such that
$\xi(q)=q^{-1}$ and $\xi(f_i)=f_{\theta(i)}$.
Let $\tp$ be an element of $\Q\otimes P$ such that
$(\tp,\alpha_i)=(\alpha_i,\al_{\theta(i)})/2$.
Define $c(\mu)=\bl((\mu+\tp,\theta(\mu+\tp))-(\tp,\theta(\tp))\br)/2+(\la,\mu)$
for $\mu\in P$.
Then it satisfies
$$c(\mu)-c(\mu-\al_i)=(\la+\mu,\al_{\theta(i)}).$$
Then we define the endomorphism $\Phi$ of $\Uf\vac_\la''$
by $\Phi(a\vac''_\la)=q^{-c(\mu)}\xi(a)\vac''_\la$ for $a\in\Uf_\mu$.
Let us show that
\eq
\Phi(F_i(a\vac''_\la))=F_i\Phi(a\vac''_\la)\quad\text{for any $a\in\Uf$.}
\label{eq:Phi}
\eneq
For $a\in\Uf_\mu$, we have 
\eqn
\Phi(F_i(a\vac''_\la))&=&
\Phi\bl(f_ia+q^{(\al_i,\la+\mu)}af_{\theta(i)}\br)\vac''_\la\\
&=&\bl(q^{-c(\mu-\al_i)}\xi(a)f_{\theta(i)}
+q^{-(\al_i,\la+\mu)-c(\mu-\al_{\theta(i)})}f_i\xi(a)\br)\vac''_\la.
\eneqn
On the other hand, we have
\eqn
F_i\Phi(a\vac''_\la)&=&
F_i\bl(q^{-c(\mu)}\xi(a)\vac''_\la\br)\\
&=&q^{-c(\mu)}\bl(f_i\xi(a)+q^{(\al_i,\la+\theta(\mu))}\xi(a)f_{\theta(i)}\br)
\vac''_\la.
\eneqn
Therefore we obtain \eqref{eq:Phi}.

Hence $\Phi$ induces the desired endomorphism of
$\Vt\subset\Uf\vac''_\la$.
\QED

Hereafter we assume further that
\[
\text{there is no $i\in I$ such that $\theta(i)=i$.}
\]
We conjecture that $V_\theta(\la)$ has a crystal basis.
This means the following.
Since $E_i$ and $F_i$ satisfy the $q$-boson relation, 
any $u\in\Vt[\la]$ can be uniquely written
as $u=\sum_{n\ge0}F_i^{(n)}u_n$ with $E_iu_n=0$.
We define the modified root operators $\tE_i$ and $\tF_i$ by:
\[
\tE_i(u)=\sum_{n\ge1}F_i^{(n-1)}u_n\ 
\text{and}\ 
\tF_i(u)=\sum_{n\ge0}F_i^{(n+1)}u_n.
\]
Let $L_\theta(\la)$ be the $\A_0$-submodule of
$V_\theta(\la)$ generated by $\tF_{i_1}\cdots\tF_{i_\ell}\vac_\la$
($\ell\ge0$ and $i_1,\ldots,i_\ell\in I$\,),
and let $\CB_\theta(\la)$ be the subset 
\[
\set{\tF_{i_1}\cdots\tF_{i_\ell}\vac_\la\bmod \qs L_\theta(\la)}%
{\text{$\ell\ge0$, $i_1,\ldots, i_\ell\in I$}}
\]
of $L_\theta(\la)/\qs L_\theta(\la)$.
\begin{conj}\label{conj:crystal}
Let $\la$ be a dominant integral weight such that $\theta(\la)=\la$.
Then we have
\begin{enumerate}
\item
$\tF_iL_\theta(\la)\subset L_\theta(\la)$
and $\tE_iL_\theta(\la)\subset L_\theta(\la)$,
\item
$\CB_\theta(\la)$ is a basis of $L_\theta(\la)/\qs L_\theta(\la)$,
\item
$\tF_i\CB_\theta(\la)\subset \CB_\theta(\la)$,
and
$\tE_i\CB_\theta(\la)\subset \CB_\theta(\la)\sqcup\{0\}$,
\item
$\tF_i\tE_i(b)=b$ for any $b\in  \CB_\theta(\la)$ such that $\tE_ib\not=0$,
and $\tE_i\tF_i(b)=b$ for any $b\in  \CB_\theta(\la)$.
\end{enumerate}
\end{conj}

As in \cite{K}, we have
\Lemma
Assume {\rm Conjecture~\ref{conj:crystal}}.
Then we have
\bnum
\item
$\Lt[\la]=\set{v\in\Vt[\la]}{(\Lt[\la],v)\subset\A_0}$,
\item
Let $(\scbul,\scbul)_0$ be the $\Q$-valued symmetric bilinear form
on $\Lt[\la]/q\Lt[\la]$ induced by $(\scbul,\scbul)$.
Then $\CB_\theta(\la)$ is an orthonormal basis with respect to
 $(\scbul,\scbul)_0$.
\enum
\enlemma

Moreover we conjecture that
$V_\theta(\la)$ has a global crystal basis.
Namely we have
\begin{conj}\label{conj:bal}
$(L_\theta(\la),L_\theta(\la)^-,\Vt[\la]^\lw_\A)$
is balanced.
Here $\Vt[\la]^\lw_\A\seteq\Uf_\A\vac_\la$.
\end{conj}
Its dual version is as follows.

Let us denote by $\Vt[\la]^{\upper}_\A$
the dual space $\set{v\in\Vt[\la]}{(\Vt[\la]^{\lw}_\A,v)\subset\A}$.
Then Conjecture~\ref{conj:bal} is equivalent to the following conjecture.
\begin{conj}\label{conj:gls}
$(L_\theta(\la),c(L_\theta(\la)),\Vt[\la]^\upper_\A)$
is balanced.
\end{conj}
Here $c$ is a unique endomorphism of $\Vt[\la]$ such that
$c(\vac_\la)=\vac_\la$ and $c(av)=\bar{a}c(v)$,
$c(E_iv)=E_ic(v)$ for any $a\in\K$ and $v\in \Vt[\la]$.
We have $(c(v'),v)=\ol{(v',\bar{v})}$ for any $v,v'\in \Vt[\la]$.
	
Note that
$\Vt[\la]^\upper_\A$ is the largest $\A$-submodule $M$ of $\Vt[\la]$ such that
$M$ is invariant by the $E_i^{(n)}$'s and
$M\cap\K \vac_\la=\A\vac_\la$.

By Conjecture~\ref{conj:gls}, 
$\Lt[\la]\cap c(\Lt[\la])\cap\Vt^\upper_\A\to\Lt[\la]/q\Lt[\la]$
is an isomorphism. Let $\Gth^\upper$ be its inverse.
Then $\{\Gth^\upper(b)\}_{b\in\Bs}$ is a basis of $\Vt[\la]$, which we call
the {\em upper global basis} of $\Vt[\la]$.
Note that $\{\Gth^\upper(b)\}_{b\in\Bs}$ is the dual basis
to $\{\Gth^\lw(b)\}_{b\in\Bs}$ with respect to the inner product of
$\Vt[\la]$.

We shall prove these conjectures in the case $\g=\gl_\infty$ and $\la=0$.

\section{PBW basis of $V_\theta(0)$ for $\mf{g}=\mf{gl}_{\infty}$}

\subsection{Review on the PBW basis}
In the sequel, we set $I=\Z_{\text{odd}}$ and 
\[
(\alpha_i,\alpha_j)=\left\{
\begin{array}{ll}
2 & \text{for $i=j$,} \\
-1 & \text{for $j=i \pm 2$,}\\
0 & \text{otherwise,}
\end{array}
\right.
\]
and we consider the corresponding
quantum group $\U[\gl_{\infty}]$. In this case,
we have $q_i=q$. We write
$[n]$ and $[n]!$ for $[n]_i$ and $[n]_i!$ for short.

We can parametrize the crystal basis $\CB(\infty)$ by the multisegments. 
We shall recall this parametrization and the PBW basis.

\begin{dfn}
For $i,j\in I$ such that $i\le j$, we define a segment $\pbw{i,j}$ 
as the interval $[i,j] \subset I\seteq\Z_{\text{odd}}$. 
A multisegment is a formal finite sum of segments:
\[
\mb=\sum_{i \le j}{m_{ij}}\lr{i}{j}
\]
with $m_{i,j}\in\Z_{\ge0}$.
We call $m_{ij}$ the multiplicity of a segment $\lr{i}{j}$. 
If $m_{i,j}>0$, we sometimes say that
$\pbw{i,j}$ appears in $\mb$.
We sometimes write
$m_{i,j}(\mb)$ for $m_{i,j}$.
We write sometimes $\pbw{i}$ for $\pbw{i,i}$.
We denote by $\mathcal{M}$ the set of multisegments.
We denote by $\emptyset$ the zero element \ro or the empty multisegment\/\rf\ 
of $\M$.
\end{dfn}
\begin{dfn}
For two segments $\lr{i_1}{j_1}$ and $\lr{i_2}{j_2}$, we define the ordering $\ge_{\text{PBW}}$ by the following:
\[
\lr{i_1}{j_1} \ge_{\text{PBW}} \lr{i_2}{j_2} \Longleftrightarrow \left\{
\begin{array}{l}
j_1>j_2 \\
\text{or} \\
j_1=j_2 \ \text{and} \ i_1 \ge i_2.
\end{array}
\right.
\]
We call this ordering the {\em PBW-ordering}.
\end{dfn}
\begin{dfn}
For a multisegment $\mb$,
we define the element $P(\mb) \in U_q^-(\mf{gl}_{\infty})$
as follows.
\be[{\rm(1)}]
\item
For a segment $\lr{i}{j}$, we define the element $\lr{i}{j}\in\Uf[\gl_\infty]$ inductively by
\begin{eqnarray*}
\lr{i}{i}&=&f_i, \\
\lr{i}{j}&=&\lr{i}{j-2}\lr{j}{j}-q\lr{j}{j}\lr{i}{j-2}\quad\text{for $i<j$.}
\end{eqnarray*}
\item
For a multisegment $\displaystyle \mb=\sum_{i \le j}m_{ij}\lr{i}{j}$, we define
\[
P(\mb)=\mathop{\overrightarrow{\prod}}
\lr{i}{j}^{(m_{ij})}.
\]
Here the product $\overrightarrow{\prod}$ 
is taken over segments appearing in $\mb$ 
from large to small with respect to the PBW-ordering.
The element $\lr{i}{j}^{(m_{ij})}$ is the divided power of $\lr{i}{j}$ i.e. 
\[
\lr{i}{j}^{(n)}=
\begin{cases}\dfrac{1}{[n]!}\lr{i}{j}^{n}&\text{for $n>0$,}\\
1&\text{for $n=0$,}\\
0&\text{for $n<0$.}
\end{cases}
\]
\ee
\end{dfn}
Hence the weight of $P(\mb)$ is equal to
$\wt(\mb)\seteq-\ssum_{i\le k\le j}m_{i,j}\al_k$:
$t_iP(\mb)t_i^{-1}=q^{(\al_i,\wt(\mb))}P(\mb)$.

\begin{thm}[\cite{L}]
The set of elements $\set{P(\mb)}{\mb \in \mathcal{M}}$ 
is a $\K$-basis of $U_q^-(\mf{gl}_{\infty})$. 
Moreover this is an $\A$-basis of $U_q^-(\mf{gl}_{\infty})_\A$. 
We call this basis the {\em PBW basis} of $U_q^-(\mf{gl}_{\infty})$. 
\end{thm}

\begin{dfn}
For two segments $\lr{i_1}{j_1}$ and $\lr{i_2}{j_2}$, 
we define the ordering $\ge_{\text{cry}}$ by the following:
\[
\lr{i_1}{j_1} \ge_{\text{cry}} \lr{i_2}{j_2} \Longleftrightarrow \left\{
\begin{array}{l}
j_1>j_2 \\
\text{or} \\
j_1=j_2 \ \text{and} \ i_1 \le i_2.
\end{array}
\right.
\]
We call this ordering the {\em crystal ordering}.
\end{dfn}

\begin{exa}
The crystal ordering is different from the PBW-ordering. For example, 
we have $\lr{-1}{1}>_{\text{cry}}\lr{1}{1}>_{\text{cry}}\lr{-1}{-1}$,
while we have $\lr{1}{1}>_{\text{PBW}}\lr{-1}{1}>_{\text{PBW}}\lr{-1}{-1}$.
\end{exa}
\begin{dfn}\label{defKop}
We define the crystal structure on $\mathcal{M}$ as follows:
for $\mb=\sum m_{i,j}\pbw{i,j}\in\mathcal{M}$ and $i\in I$,
set $A_k^{(i)}(\mb)=\sum_{k'\ge k}(m_{i,k'}-m_{i+2,k'+2})$ for $k\ge i$.
Define $\eps_i(\mb)$ as $\max\set{A_k^{(i)}(\mb)}{k\ge i}\ge0$.
\bnum
\item If $\eps_i(\mb)=0$, then define
$\te_i(\mb)=0$.
If $\eps_i(\mb)>0$, let
$k_e$ be the largest $k\ge i$
such that $\eps_i(\mb)=A_k^{(i)}(\mb)$
and define
$\te_i(\mb)=\mb-\pbw{i,k_e}+\delta_{k_e\not=i}\pbw{i+2,k_e}$.
\item
Let $k_f$ be the smallest $k\ge i$
such that $\eps_i(\mb)=A_k^{(i)}(\mb)$
and define
$\tf_i(\mb)=\mb-\delta_{k_f\not=i}\pbw{i+2,k_f}+\pbw{i,k_f}$.
\enum
\end{dfn}
\begin{rem}
For $i \in I$,
the actions of the operators $\widetilde{e}_i$ and $\widetilde{f}_i$ on 
$\mb\in\mathcal{M}$ 
are also described by the following algorithm: 

\be[{Step 1.}]
\item Arrange the segments in $\mb$ in the crystal ordering. 
\item
For each segment $\lr{i}{j}$, write $-$, 
and for each segment $\lr{i+2}{j}$, write $+$. 
\item In the resulting sequence of $+$ and $-$, delete a subsequence of the form $+-$ and keep on deleting until no such subsequence remains. 
\ee
Then we obtain a sequence of the form $-- \cdots -++ \cdots +$. 

\be[{\rm(1)}]
\item
$\varepsilon_i(\mb)$ is the total number of $-$ 
in the resulting sequence. 
\item $\widetilde{f}_i(\mb)$ is given as follows:
\be[{(a)}]
\item
if the leftmost $+$ corresponds to a segment $\lr{i+2}{j}$, 
then replace it with $\lr{i}{j}$,

\item if no $+$ exists, add a segment $\lr{i}{i}$ to $\mb$.
\ee

\item $\widetilde{e}_i(\mb)$  is given as follows:

\be[{(a)}]
\item
if the rightmost $-$ corresponds to a segment $\lr{i}{j}$, 
then replace it with $\lr{i+2}{j}$,
\item if no $-$ exists, then $\widetilde{e}_i(\mb)=0$. 
\ee
\ee
\end{rem}

Let us introduce a linear ordering on the set $\M$ of multisegments,
lexicographic with respect to the crystal ordering on the set of segments.
\begin{dfn}
For $\mb=\sum_{i\le j}m_{i,j}\pbw{i,j}\in \M$ and
and $\mb'=\sum_{i\le j}m'_{i,j}\pbw{i,j}\in \M$,
we define $\mb'\ltcr\mb$
if there exist $i_0\le j_0$ such that
$m'_{i_0,j_0}<m_{i_0,j_0}$,
$m'_{i,j_0}=m_{i,j_0}$ for $i<i_0$, and
$m'_{i,j}=m_{i,j}$ for $j>j_0$ and $i\le j$.
\end{dfn}
\begin{thm}\label{thmgl}
\bnum
\item
 $\displaystyle L(\infty)=\soplus_{\mb \in \mathcal{M}}\mathbf{A}_0P(\mb)$. 
\item
$\CB(\infty)
=\set{P(\mb) \mod qL(\infty)}{\mb \in \M}$. 
\item
We have
\begin{eqnarray*}
\widetilde{e}_iP(\mb) &\equiv& 
P(\widetilde{e}_i(\mb)) \quad \mod{qL(\infty)}, \\
\widetilde{f}_iP(\mb) &\equiv& P(\widetilde{f}_i(\mb)) 
\quad \mod{qL(\infty)}.
\end{eqnarray*}
Note that $\widetilde{e}_i$ and $\widetilde{f}_i$ 
in the left-hand-side is the modified root operators.
\item
We have
\[
\overline{P(\mb)}
\in P(\mb)+\sum_{\mb'\ltcr \mb}\A P(\mb').
\]
\ee
\end{thm}
Therefore we can index the crystal basis by multisegments.
By this theorem we can easily see by a standard argument that
$(L(\infty),L(\infty)^-,\Uf_\A)$ is balanced,
and there exists a unique $\Gl^\lw(\mb)\in L(\infty)\cap\Uf_\A$ 
such that $\Gl^\lw(\mb)^{-}=\Gl^\lw(\mb)$ and 
$\Gl^\lw(\mb)\equiv P(\mb) \mod{qL(\infty)}$.
The basis $\{\Gl^\lw(\mb)\}_{\mb\in\M}$ is a lower global basis.

\subsection{$\theta$-restricted multisegments}

We consider 
the Dynkin diagram involution $\theta$ of $I$ defined by 
$\theta(i)=-i$ for $i \in I=\Z_{\text{odd}}$. 
{\scriptsize
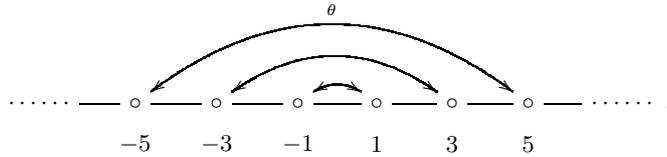
$$
\xymatrix@R=.8ex@C=3ex{
\cdots\cdots\ar@{-}[r]&\circ\ar@{-}[r]
\ar@/^2.5pc/@{<->}[rrrrr]^\theta
&\circ\ar@{-}[r]\ar@/^1.5pc/@{<->}[rrr]&\circ\ar@{-}[r]
\ar@/^.6pc/@{<->}[r]&
\circ\ar@{-}[r]&\circ\ar@{-}[r]&\circ\ar@{-}[r]&\cdots\cdots\ .\\
&-5&-3&-1&\;1\;&\;3\;&\;5\;
}
$$}
We shall prove in this case Conjectures \ref{conj:crystal}
and \ref{conj:bal} for $\la=0$ (Theorems~\ref{main:cr} and \ref{th:gl}).

We set
\eqn
\tVt&\seteq&\Btg/
\ssum_{i\in I}\bl(\Btg E_i+\Btg(T_i-1)+\Btg(F_i-F_{\theta(i)})\br)\\
&\simeq&
U_q^-(\mf{gl}_{\infty})/\ssum_{i}U_q^-(\mf{gl}_{\infty})(f_i-f_{\theta(i)}).
\eneqn
Let $\tvac$ be the generator of $\tVt$ corresponding to $1\in\Btg$.
Since $F_i\vac_{0}''=(f_i+f_{\theta(i)})\vac_0''=F_{\theta(i)}\vac_0''$,
we have an epimorphism of $\Btg$-modules
\eq
&&\tVt\epi V_\theta(0).
\eneq

We shall see later that it is in fact an isomorphism
(see Theorem~\ref{main:cr}).

\begin{dfn}
If a multisegment $\mb$ has the form
\[
\mb=\sum_{-j \le i \le j}m_{ij}\lr{i}{j},
\]
we call $\mb$ a {\em $\theta$-restricted} multisegment. 
We denote by $\mathcal{M}_\theta$ the set of $\theta$-restricted multisegments.
\end{dfn}

\begin{dfn}
 For a $\theta$-restricted segment $\pbw{i,j}$,
we define its modified divided power by
\[
{\lr{i}{j}}^{\dv{m}}=\left\{
\begin{array}{ll}
\pbw{i,j}^{(m)}=\dfrac{1}{[m]!}\lr{i}{j}^m & (i \neq -j), \\
\dfrac{1}{\prod_{\nu=1}^{m}[2\nu]}\lr{-j}{j}^m & (i=-j).
\end{array}
\right.
\]
\end{dfn}
We understand that
$\pbw{i,j}^{\dv{m}}$ is equal to $1$ for $m=0$ and vanishes for $m<0$.

\begin{dfn}
For $\mb \in \mathcal{M}_\theta$, 
we define the element $P_\theta(\mb) \in \Uf[\gl_\infty]\subset\Btg$ by 
\[
P_\theta(\mb)=\mathop{\overrightarrow{\prod}}\limits
_{\lr{i}{j} \in \mb}\lr{i}{j}^{\dv{m_{ij}}}.
\]
Here the product $\overrightarrow{\prod}$ is taken 
over the segments
appearing in $\mb$ from large to small
with respect to the PBW-ordering. 

 If an element $\mb$ of the free abelian group generated by
$\pbw{i,j}$ does not belong to $\Mt$, we understand
$\Pt(\mb)=0$.
\end{dfn}
We will prove later
that $\{P_\theta(\mb)\vac\}_{\mb \in \Mt}$ 
is a basis of $\Vt$ (see Theorem~\ref{main:cr}). 
Here and hereafter, we write $\vac$ instead of $\vac_0\in\Vt$.

\subsection{Commutation relations of $\lr{i}{j}$}
In the sequel, we regard
$\Uf[\gl_\infty]$ as a subalgebra of $\Btg$
by $f_i\mapsto F_i$.

We shall give formulas to express products of segments
by a PBW basis.
\begin{prop}\label{pp1} 
For $i,j,k,l\in I$, we have
\begin{enumerate}[{\rm(1)}]
\item
$\pbw{i,j}\pbw{k,\ell}=\pbw{k,\ell}\pbw{i,j}$ for
$i\le j$, $k\le \ell$ and $j<k-2$,
\label{sub0}
\item
$\pbw{i,j}\pbw{j+2,k}=\pbw{i,k}+q\pbw{j+2,k}\pbw{i,j}$ for $i \le j<k$,
\label{sub1}
\item
$\lr{j}{k}\lr{i}{\ell}=
\lr{i}{\ell}\lr{j}{k}$ for $i<j \le k<\ell$,
\label{sub2}
\item
$\lr{i}{k}\lr{j}{k}=q^{-1}
\langle j,k \rangle \langle i,k \rangle$
for $i<j \le k$,
\label{sub3}
\item
$\langle i,j \rangle \langle i,k
\rangle=q^{-1}
\langle i,k \rangle \langle i,j \rangle$ for $i \le j<k$,
\label{sub4}
\item
$\langle i,k \rangle \langle
j,\ell
\rangle=\pbw{j,\ell}\pbw{i,k}+(q^{-1}-q) \langle i,\ell \rangle \langle j,k
\rangle$ for $i<j \le k<\ell$.
\label{sub5}
\end{enumerate}
\end{prop}
\begin{proof}
\eqref{sub0} is obvious.
We prove \eqref{sub1} by the induction on $k-j$. 
If $k-j=2$, it is trivial by
the
definition. If $j<k-2$, then $\langle k \rangle$ and $\langle
i,j \rangle$ commute. Thus, we have
\begin{eqnarray*}
\pbw{i,j}\pbw{j+2,k}
&=&\pbw{i,j}\bl(\pbw{j+2,k-2}\pbw{k}-q\pbw{k}\pbw{j+2,k-2}\br)\\
&=&\bl(\pbw{i,k-2}+q\pbw{j+2,k-2}\pbw{i,j}\br)\pbw{k}
-q\pbw{k}\pbw{i,j}\pbw{j+2,k-2}\\
&=&\pbw{i,k-2}\pbw{k}+q\pbw{j+2,k-2}\pbw{k}\pbw{i,j}\\
&&\hs{15ex}-q\pbw{k}\bl(\pbw{i,k-2}+q\pbw{j+2,k-2}\pbw{i,j}\br)\\
&=&\pbw{i,k}+\pbw{j+2,k}\pbw{i,j}.
\end{eqnarray*}
In order to prove the other relations, we first show the following
special cases.
\begin{lem}\label{lem:sp}
We have for any $j \in I$
\be[{\rm(a)}]
\item
$\pbw{j-2,j}\pbw{j}=q^{-1}\pbw{j}\pbw{j-2,j}$
and $\pbw{j}\pbw{j,j+2}=q^{-1}\pbw{j,j+2}\pbw{j}$,
\item
$
\langle j\rangle \langle j-2,j+2 \rangle
=\langle j-2,j+2 \rangle \langle j \rangle
$,
\item
$\pbw{j-2,j}\pbw{j,j+2}=\pbw{j,j+2}\pbw{j-2,j}+(q^{-1}-q)
\pbw{j-2,j+2}\pbw{j}$.
\ee
\end{lem}
\Proof
The first equality in (a) follows from
\begin{eqnarray*}
\langle j-2,j \rangle \langle j \rangle-q^{-1}\langle
j \rangle
\langle j-2,j
\rangle&=&\bl(f_{j-2}f_j-qf_{j}f_{j-2}\br)f_j-
q^{-1}f_j\bl(f_{j-2}f_j-qf_{j}f_{j-2}\br)\\
&=&f_{j-2}f_j^2-(q+q^{-1})f_jf_{j-2}f_j+f_j^2f_{j-2}=0.
\end{eqnarray*}

We can similarly prove the second.

Let us show (b) and (c).
We have, by (a)
\eq
&&\ba{l}
\pbw{j-2,j}\pbw{j,j+2}=
\pbw{j-2,j}\bl(\pbw{j}\pbw{j+2}-q\pbw{j+2}\pbw{j}\br)\\[1ex]
\quad=q^{-1}\pbw{j}\pbw{j-2,j}\pbw{j+2}
-q\bl(\pbw{j-2,j+2}+q\pbw{j+2}\pbw{j-2,j}\br)\pbw{j}\\[1ex]
\quad=q^{-1}\pbw{j}\bl(\pbw{j-2,j+2}+q\pbw{j+2}\pbw{j-2,j}\br)\\
\hs{8ex} -q\pbw{j-2,j+2}\pbw{j}-q\pbw{j+2}\pbw{j}\pbw{j-2,j}\\[1ex]
\quad=\bl(\pbw{j}\pbw{j+2}-q\pbw{j+2}\pbw{j}\br)\pbw{j-2,j}
+q^{-1}\pbw{j}\pbw{j-2,j+2}-q\pbw{j-2,j+2}\pbw{j}\\[1ex]
\quad=\pbw{j,j+2}\pbw{j-2,j}+q^{-1}\pbw{j}\pbw{j-2,j+2}-q\pbw{j-2,j+2}\pbw{j}.
\ea\label{eq:-202a}
\eneq
Similarly, we have
\eq
&&\ba{l}
\pbw{j-2,j}\pbw{j,j+2}=
\bl(\pbw{j-2}\pbw{j}-q\pbw{j}\pbw{j-2}\br)
\pbw{j,j+2}\\[1ex]
\quad=q^{-1}\pbw{j-2}\pbw{j,j+2}\pbw{j}
-q\pbw{j}\bl(\pbw{j-2,j+2}+q\pbw{j,j+2}\pbw{j-2}\br)\\[1ex]
\quad=q^{-1}\bl(\pbw{j-2,j+2}+q\pbw{j,j+2}\pbw{j-2}\br)\pbw{j}\\
\hs{8ex}-q\pbw{j}\pbw{j-2,j+2}-q\pbw{j,j+2}\pbw{j}\pbw{j-2}\\[1ex]
\quad=\pbw{j,j+2}\bl(\pbw{j-2}\pbw{j}-q\pbw{j}\pbw{j-2}\br)
+q^{-1}\pbw{j-2,j+2}\pbw{j}-q\pbw{j}\pbw{j-2,j+2}\\[1ex]
\quad=\pbw{j,j+2}\pbw{j-2,j}+q^{-1}\pbw{j-2,j+2}\pbw{j}-q\pbw{j}\pbw{j-2,j+2}.
\ea\label{eq:-202b}
\eneq
Then, \eqref{eq:-202a} and \eqref{eq:-202b} imply (b) and (c).
\QED
We shall resume the proof of Proposition~\ref{pp1}.
By Lemma~\ref{lem:sp} (b),  $\langle i,k \rangle$
commutes with $\pbw{j}$ for
$i<j<k$. Thus we obtain \eqref{sub2}.

We shall show \eqref{sub3} by the induction on $k-j$. Suppose $k-j=0$. 
The case $i=k-2$ is nothing but  Lemma~\ref{lem:sp} (a).

If $i<k-2$, then
\begin{eqnarray*}
\langle i,k \rangle \langle k \rangle
&=&\langle i,k-4 \rangle \langle k-2,k \rangle
\langle k
\rangle-q\langle k-2,k \rangle \langle i,k-4 \rangle
\langle k \rangle \\
&=&q^{-1}\langle k \rangle \langle i,k-4 \rangle \langle
k-2,k
\rangle-\langle k \rangle \langle k-2,k \rangle \langle
i,k-4 \rangle
=q^{-1}\langle k
\rangle \langle
i,k \rangle.
\end{eqnarray*}
Suppose $k-j>0$. By using the induction hypothesis and \eqref{sub2},
we have
\begin{eqnarray*}
\pbw{i,k}\pbw{j,k}&=&\pbw{i,k}\pbw{j}\pbw{j+2,k}-q\pbw{i,k}\pbw{j+2,k}\pbw{j}\\
&=&\pbw{j}\pbw{i,k}\pbw{j+2,k}-\pbw{j+2,k}\pbw{i,k}\pbw{j}\\
&=&q^{-1}\langle j \rangle \langle j+2,k \rangle \langle
i,k
\rangle-\langle j+2,k \rangle \langle j \rangle \langle
i,k \rangle 
=q^{-1}\langle j,k \rangle \langle i,k \rangle.
\end{eqnarray*}
Similarly we can prove \eqref{sub4}. 

Let us prove \eqref{sub5}.
We have
\eqn
\pbw{i,k}\pbw{j,\ell}
&=&\bl(\pbw{i,j-2}\pbw{j,k}-q\pbw{j,k}\pbw{i,j-2}\br)\pbw{j,\ell}\\
&=&q^{-1}\pbw{i,j-2}\pbw{j,\ell}\pbw{j,k}-
q\pbw{j,k}\bl(\pbw{i,\ell}+q\pbw{j,\ell}\pbw{i,j-2}\br)\\
&=&q^{-1}\bl(\pbw{i,\ell}+q\pbw{j,\ell}\pbw{i,j-2}\br)\pbw{j,k}-
q\pbw{i,\ell}\pbw{j,k}-q\pbw{j,\ell}\pbw{j,k}\pbw{i,j-2}\\
&=&\pbw{j,\ell}\pbw{i,k}
+(q^{-1}-q)\pbw{i,\ell}\pbw{j,k}.
\eneqn
\QED

\Lemma\label{lemq-1}
\bnum
\item For $1\le i\le j$, we have
$\pbw{-j,-i}\tvac=\pbw{i,j}\tvac$.
\item
For $1\le i<j$, we have
$\pbw{-j,i}\tvac=q^{-1}\pbw{-i,j}\tvac$.
\enum
\enlemma
\Proof
(i)\quad If $i=j$, it is obvious. By the induction on $j-i$,
we have
\eqn
\pbw{-j,-i}\tvac&=&(\pbw{-j,-i-2}\pbw{-i}-q\pbw{-i}\pbw{-j,-i-2})\tvac\\
&=&(\pbw{-j,-i-2}\pbw{i}-q\pbw{-i}\pbw{i+2,j})\tvac\\
&=&(\pbw{i}\pbw{-j,-i-2}-q\pbw{i+2,j}\pbw{-i})\tvac\\
&=&(\pbw{i}\pbw{i+2,j}-q\pbw{i+2,j}\pbw{i})\tvac
=\pbw{i,j}\tvac.
\eneqn

\noindent
(ii)\quad By (i), we have
\eqn
\pbw{-j,i}\tvac
&=&(\pbw{-j,-1}\pbw{1,i}-q\pbw{1,i}\pbw{-j,-1})\tvac\\
&=&(\pbw{-j,-1}\pbw{-i,-1}-q\pbw{1,i}\pbw{1,j})\tvac\\
&=&(q^{-1}\pbw{-i,-1}\pbw{-j,-1}-\pbw{1,j}\pbw{1,i})\tvac\\
&=&(q^{-1}\pbw{-i,-1}\pbw{1,j}-\pbw{1,j}\pbw{-i,-1})\tvac
=q^{-1}\pbw{-i,j}\tvac.
\eneqn
\QED

\begin{prop}\label{pp2}
\bnum
\item For a multisegment $\mb=\sum_{i\le j}m_{i,j}\pbw{i,j}$, we have
$$\Ad(t_k)P(\mb)
=q^{\sum_i(m_{i,k-2}-m_{i,k})+\sum_j(m_{k+2,j}-m_{k,j})}
P(\mb).$$
\item
\begin{eqnarray*}
e_k'\langle i,j \rangle^{(n)}&=&
\begin{cases}
q^{1-n}\langle i \rangle^{(n-1)} & \text{if $k=i=j$,} \\
(1-q^2)q^{1-n}\langle i+2,j \rangle \langle i,j
\rangle^{(n-1)} & \text{if $k=i<j$,} \\
0 & \text{otherwise,}
\end{cases}
\\
e_k^*\langle i,j \rangle^{(n)}&=&
\begin{cases}
q^{1-n}\langle i \rangle^{(n-1)} & \text{if $i=j=k$,} \\
(1-q^2)q^{1-n}\langle i,j \rangle^{(n-1)} \langle i,j-2
\rangle & \text{if $i<j=k$,} \\
0 & \text{otherwise.} \\
\end{cases}
\end{eqnarray*}
\ee
\end{prop}
\begin{proof}
(i) is obvious. Let us show (ii).
It is obvious that
$e'_k\pbw{i,j}^{(n)}=0$ unless
$i\le k\le j$.
It is known (\cite{K}) that we have
$e'_k\pbw{k}^{(n)}=q^{1-n}\langle k\rangle^{(n-1)}$. 
%
We shall prove $e_k'\langle k,j
\rangle^{(n)}=(1-q^2)q^{1-n}\langle
k+2,j\rangle \langle k,j \rangle^{(n-1)}$ for $k<j$ by the induction on
$n$. By \eqref{eq:der}, we have 
\begin{eqnarray*}
e_k'\langle k,j \rangle&=&e_k'(\langle k \rangle \langle
k+2,j \rangle-q
\langle k+2,j \rangle \langle k \rangle) \\
&=&\pbw{k+2,j}-q^2\pbw{k+2,j}=(1-q^2)\pbw{k+2,j}.
\eneqn
%
For $n\ge 1$, by the induction hypothesis and Proposition \ref{pp1} 
\eqref{sub3}, we get
\begin{eqnarray*}
&&[n]e_k'\langle k,j \rangle^{(n)} 
=e'_k\pbw{k,j}\pbw{k,j}^{(n-1)}\\
&&\quad=(1-q^2)\langle k+2,j \rangle \langle k,j
\left. \rangle^{(n-1)}+q^{-1}\langle k,j \rangle \cdot \right.
(1-q^2)q^{2-n}\langle k+2,j \rangle \langle k,j
\rangle^{(n-2)} \\
&&\quad=(1-q^2)\left\{\langle k+2,j
\rangle \langle
k,j \rangle^{(n-1)}
+q^{1-n}\langle k,j\rangle \langle
k+2,j \rangle \langle k,j \rangle^{(n-2)}\right\}\\
&&\quad=(1-q^2)(1+q^{-n}[n-1])\langle k+2,j\rangle \langle k,j
\rangle^{(n-1)}\\
&&\quad=(1-q^2)q^{1-n}[n]\langle k+2,j\rangle \langle k,j
\rangle^{(n-1)}.
\end{eqnarray*}
Finally we show $e_k'\langle i,j \rangle=0$ if $k \neq i$.
We may assume $i<k\le j$.
If $i<k<j$, we have
\begin{eqnarray*}
e_k'\langle i,j \rangle&=&
e'_k(\pbw{i,k-2}\pbw{k,j}-q\pbw{k,j}\pbw{i,k-2})\\
&=&
q\pbw{i,k-2}e'_k\pbw{k,j}-q(e'_k\pbw{k,j})\pbw{i,k-2}\\
&=&q(1-q^2)\langle i,k-2 \rangle \langle k+2,j
\rangle-q(1-q^2)\langle
k+2,j \rangle \langle i,k-2 \rangle \\
&=&0.
\end{eqnarray*}
The case $k=j$ is similarly proved.

The proof for $e_k^*$ is similar.
\end{proof}

\subsection{Actions of divided powers}

\begin{lem}\label{lem1}
Let $a$, $b$ be non-negative integers,
and let $k\in I_{>0}\seteq\set{k\in I}{k>0}$.
\begin{enumerate}[{\rm(1)}]
\item{For $\ell	>k$, we have
\begin{eqnarray*}
\langle -k \rangle \langle -k+2,\ell \rangle^{(a)} \langle
-k,\ell
\rangle^{(b)}&=&[b+1]\langle -k+2,\ell
\rangle^{(a-1)}\langle -k,\ell
\rangle^{(b+1)} \\
&{}&+q^{a-b}\langle -k+2,\ell
\rangle^{(a)}\langle -k,\ell\rangle^{(b)}\langle -k
\rangle.
\end{eqnarray*}}
\item{We have
\begin{eqnarray*}
\langle -k \rangle \langle -k+2,k \rangle^{(a)}\langle
-k,k
\rangle^{\dv{b}}&=&[2b+2]\langle -k+2,k\rangle^{(a-1)}\langle
-k,k
\rangle^{\dv{b+1}} \\
&{}&+q^{a-b}\langle -k+2,k \rangle^{(a)}\langle -k,k
\rangle^{\dv{b}}\langle -k
\rangle.
\end{eqnarray*}
}
\item{For $k>1$, we have
\begin{eqnarray*}
\langle -k \rangle \langle -k+2,k-2
\rangle^{\dv{a}}&=&(q^a+q^{-a})^{-1}\langle -k+2,k-2
\rangle^{\dv{a-1}}\langle -k,k-2 \rangle \\
&{}&+q^a\langle -k+2,k-2 \rangle^{\dv{a}}\langle -k \rangle.
\end{eqnarray*}
}
\item{If $\ell\le k-2$, we have
\begin{eqnarray*}
\langle \ell,k-2 \rangle^{(a)}\langle k \rangle=\langle
\ell,k \rangle
\langle \ell,k-2 \rangle^{(a-1)}+q^a\langle k \rangle
\langle \ell,k-2
\rangle^{(a)}.
\end{eqnarray*}
}
\item{For $k>1$, we have
\begin{eqnarray*}
\langle -k+2,k-2 \rangle^{\dv{a}}\langle k
\rangle&=&(q^a+q^{-a})^{-1}\langle -k+2,k \rangle
\langle -k+2,k-2\rangle^{\dv{a-1}} \\
&{}&+q^a\langle k \rangle \langle -k+2,k-2 \rangle^{\dv{a}}.
\end{eqnarray*}
}
\end{enumerate}
\end{lem}
\begin{proof}
We show (1) by the induction on $a$. If $a=0$, it is trivial.
For $a>0$, we have
\begin{eqnarray*}
&{}&[a]\langle -k \rangle \langle -k+2,\ell \rangle^{(a)}
\langle -k,\ell
\rangle^{(b)} \\
&=&\bl(\langle -k,\ell \rangle+q\langle
-k+2,\ell \rangle
\langle -k \rangle\br)\langle -k+2,\ell
\rangle^{(a-1)}\langle -k,\ell
\rangle^{(b)} \\
&=&[b+1]q^{1-a}\langle -k+2,\ell
\rangle^{(a-1)}\langle -k,\ell \rangle^{(b+1)} \\
&{}&+q\langle -k+2,\ell \rangle \{[b+1]\langle -k+2,\ell
\rangle^{(a-2)}\langle -k,\ell
\rangle^{(b+1)}+q^{a-b-1}\langle -k+2,\ell
\rangle^{(a-1)}\langle -k,\ell \rangle^{(b)}\langle -k
\rangle\} \\
&=&[b+1](q^{1-a}+q[a-1])
\langle -k+2,\ell \rangle^{(a-1)}\langle -k,\ell
\rangle^{(b+1)}+q^{a-b}[a]\langle -k+2,\ell
\rangle^{(a)}\langle -k,\ell\rangle^{(b)}\langle -k
\rangle.
\end{eqnarray*}
Since $q^{1-a}+q[a-1]=[a]$, the induction proceeds.

The proof of (2) is similar by using
$\langle -k,k \rangle^{\dv{b}}=
[2b]\langle -k,k \rangle^{\dv{b-1}}\pbw{-k,k}$.

We prove (3) by the induction on $a$. The case $a=0$ is trivial.
For $a>0$, we have
\begin{eqnarray*}
&&[2a]\langle -k \rangle \langle -k+2,k-2 \rangle^{\dv{a}}
=\bl(\pbw{-k,k-2}+q\pbw{-k+2,k-2}\pbw{-k}\br)\pbw{-k+2,k-2}^{\dv{a-1}}\\
&&\quad= q^{1-a}\langle -k+2,k-2 \rangle^{\dv{a-1}}
\langle -k,k-2 \rangle \\
&&\qquad+{q}\langle -k+2,k-2 \rangle \bl\{
(q^{a-1}+q^{1-a})^{-1}\langle -k+2,k-2 \rangle^{\dv{a-2}}
\langle -k,k-2 \rangle \\
&&\hs{30ex}+q^{a-1}\langle -k+2,k-2 \rangle^{\dv{a-1}} 
\langle -k \rangle\br\} \\
&=&\bl(q^{1-a}+\dfrac{q[2a-2]}{q^{a-1}+q^{1-a}}\br)
\langle -k+2,k-2 \rangle^{\dv{a-1}}\langle -k,k-2 \rangle+q^a[2a]
\langle -k+2,k-2 \rangle^{\dv{a}}\langle -k \rangle\\
&=&(q^a+q^{-a})^{-1}[2a]\langle -k+2,k-2 \rangle^{\dv{a-1}}
\langle -k,k-2 \rangle+q^a[2a]\langle -k+2,k-2 \rangle^{\dv{a}}
\langle -k \rangle.
\end{eqnarray*}
Similarly, we can prove (4) and (5) by the induction on $a$.
\end{proof}

\Lemma\label{lem:ijta}
For $k>1$  and $a,b,c,d\ge0$, set
\[
(a,b,c,d)=\langle k \rangle^{(a)}\lr{-k+2}{k}^{(b)}
\lr{-k}{k}^{\dv{c}}\lr{-k+2}{k-2}^{\dv{d}}\tvac.
\]
Then, we have
\eq
&&\ba{rcl}
\pbw{-k}(a,b,c,d)&=&
[2c+2](a,b-1,c+1,d) \\[1ex]
&&+[b+1]q^{b-2c}(a,b+1,c,d-1)\\[1ex]
&&+[a+1]q^{2d-2c}(a+1,b,c,d).
\ea
\label{eq:-kijta}
\eneq
\enlemma
\Proof
We shall show first
\eq
&&
\ba{l}\pbw{-k}\pbw{-k+2,k-2}^{\dv{d}}\tvac\\[1ex]
\hs{10ex}=\bl(\pbw{-k+2,k}\pbw{-k+2,k-2}^{\dv{d-1}}
+q^{2d}\pbw{k}\pbw{-k+2,k-2}^{\dv{d}}\br)\tvac.
\ea\label{eq:45}
\eneq
By Lemma~\ref{lem1} (3),
we have
\eqn
\pbw{-k}\pbw{-k+2,k-2}^{\dv{d}}\tvac&=&
\bl((q^d+q^{-d})^{-1}\pbw{-k+2,k-2}^{\dv{d-1}}\pbw{-k,k-2}\\
&&\hs{15ex}+q^d\pbw{-k+2,k-2}^{\dv{d}}\pbw{-k}\br)\tvac.
\eneqn
By Lemma~\ref{lemq-1} and Lemma~\ref{lem1} (5), it is equal to
\eqn
&&\bl((q^d+q^{-d})^{-1}q^{-1}\pbw{-k+2,k-2}^{\dv{d-1}}\pbw{-k+2,k}
+q^d\pbw{-k+2,k-2}^{\dv{d}}\pbw{k}\br)\tvac\\
&&=\Bigl((q^d+q^{-d})^{-1}q^{-1}q^{1-d}\pbw{-k+2,k}\pbw{-k+2,k-2}^{\dv{d-1}}\\
&&\hs{5ex}+q^d\bigl((q^d+q^{-d})^{-1}\pbw{-k+2,k}\pbw{-k+2,k-2}^{\dv{d-1}}
+q^d\pbw{k}\pbw{-k+2,k-2}^{\dv{d}}\br)\Bigr)\tvac.
\eneqn
Thus we obtain \eqref{eq:45}.
Applying Lemma \ref{lem1} (2), we have
\eqn
&&\pbw{-k}(a,b,c,d)
=\pbw{k}^{(a)}\Bigl([2c+2]\pbw{-k+2,k}^{(b-1)}\pbw{-k,k}^{\dv{c+1}}\\
&&\hs{20ex}+q^{b-c}\pbw{-k+2,k}^{(b)}\pbw{-k,k}^{\dv{c}}\pbw{-k}\Bigr)
\pbw{-k+2,k-2}^{\dv{d}}\tvac\\
&&\qquad=
[2c+2](a,b-1,c+1,d)+
q^{b-c}\pbw{k}^{(a)}\pbw{-k+2,k}^{(b)}\pbw{-k,k}^{\dv{c}}\\
&&\hs{10ex}\times\bl(\pbw{-k+2,k}\pbw{-k+2,k-2}^{\dv{d-1}}
+q^{2d}\pbw{k}\pbw{-k+2,k-2}^{\dv{d}}\br)\tvac\\
&&\qquad=[2c+2](a,b-1,c+1,d)+q^{b-2c}[b+1](a,b+1,c,d-1)\\
&&\hs{40ex}+q^{(b-c)+2d-c-b}[a+1](a+1,b,c,d).
\eneqn
Hence we have \eqref{eq:-kijta}.
\QED

\begin{prop} \label{prop:div}
\be[{\rm(1)}]
\item We have
\begin{eqnarray*}
\langle -1 \rangle^{(a)} \lr{-1}{1}^{\dv{m}}\tvac
=\sum_{s=0}^{\lfloor{a/2}\rfloor}\left(\prod_{\nu=1}^{s}
\dfrac{[2m+2\nu]}{[2\nu]}\right)q^{-2(a-s)m
+\frac{(a-2s)(a-2s-1)}{2}}\langle 1 \rangle^{(a-2s)}
\lr{-1}{1}^{\dv{m+s}}\tvac.
\end{eqnarray*}
\item{For $k>1$, we have
\begin{eqnarray*}
&{}&\langle -k \rangle^{(n)}\lr{-k+2}{k-2}^{\dv{a}}\tvac \\
&=&\sum_{u=0}^{n}\sum_{i+j+2t=n,j+t=u}q^{2ai+\frac{j(j-1)}{2}-i(t+u)}
\pbw{k}^{(i)}\lr{-k+2}{k}^{(j)}\lr{-k}{k}^{\dv{t}}
\lr{-k+2}{k-2}^{\dv{a-u}}\tvac.
\end{eqnarray*}
}
\item{If $\ell>k$, we have
\begin{eqnarray*}
\langle k \rangle^{(n)}\lr{k+2}{\ell}^{(a)}
=\sum_{s=0}^{n}q^{(n-s)(a-s)}
\lr{k+2}{\ell}^{(a-s)}\lr{k}{\ell}^{(s)}\langle k \rangle^{(n-s)}.
\end{eqnarray*}
}
\end{enumerate}
\end{prop}
\begin{proof}
We prove (1) by the induction on $a$. 
The case $a=0$ is trivial.
Assume $a>0$. Then, Lemma~\ref{lem1} (2) implies 
\eqn
\pbw{-1}\langle 1 \rangle^{(n)} \lr{-1}{1}^{\dv{m}}\tvac
&=&\bl([2m+2]\langle 1 \rangle^{(n-1)}\lr{-1}{1}^{\dv{m+1}}
+q^{n-m}\langle 1 \rangle^{(n)}\lr{-1}{1}^{\dv{m}}\pbw{-1}\br)
\tvac\\
&=&\bl([2m+2]\langle 1 \rangle^{(n-1)}\lr{-1}{1}^{\dv{m+1}}
+q^{n-m}\langle 1 \rangle^{(n)}\lr{-1}{1}^{\dv{m}}\pbw{1}\br)
\tvac\\
&=&\bl(
[2m+2]\langle 1 \rangle^{(n-1)}\lr{-1}{1}^{\dv{m+1}}
+q^{n-2m}[n+1]\langle 1 \rangle^{(n+1)}\lr{-1}{1}^{\dv{m}}\br)\tvac.
\eneqn
Put
\[
c_s=\left(\prod_{\nu=1}^{s}\dfrac{[2m+2\nu]}{[2\nu]}\right)q^{-2(a-s)m+\frac{(a-2s)(a-2s-1)}{2}}.
\]
Then we have
\begin{eqnarray*}
&{}&[a+1]\langle -1 \rangle^{(a+1)}\lr{-1}{1}^{\dv{m}}\tvac 
=\pbw{-1}\langle -1 \rangle^{(a)}\lr{-1}{1}^{\dv{m}}\tvac \\
&&\hs{10ex}=\langle -1 \rangle \sum_{s=0}^{\lfloor{a/2}\rfloor}c_s
\langle 1 \rangle^{(a-2s)}\lr{-1}{1}^{\dv{m+s}}\tvac \\
&&\hs{10ex}=\sum_{s=0}^{\lfloor{a/2}\rfloor}c_s
\bl\{[2(m+s+1)]\langle 1 \rangle^{(a-2s-1)}
\lr{-1}{1}^{\dv{m+s+1}}\\
&&\hs{20ex}+q^{a-2s-2(m+s)}[a-2s+1]\langle 1 \rangle^{(a-2s+1)}
\lr{-1}{1}^{\dv{m+s}}\br\}\tvac.
\end{eqnarray*}
In the right-hand-side, the coefficients of 
$\langle 1 \rangle^{a+1-2r}\lr{-1}{1}^{\dv{ m+r}}\tvac$ are
\begin{eqnarray*}
&{}&[2(m+r)]c_{r-1}+q^{a-2m-4r}[a-2r+1]c_r \\
&&=\prod_{\nu=1}^{r}\dfrac{[2m+2\nu]}{[2\nu]}
q^{-2(a-r+1)m+\frac{(a-2r)(a-2r+1)}{2}} 
\Bigl([2r]q^{a-2r+1}+[a-2r+1]q^{-2r}\Bigr) \\
&&=[a+1]\prod_{\nu=1}^{r}\dfrac{[2m+2\nu]}{[2\nu]}q^{-2(a-r+1)m+
\frac{(a-2r)((a-2r+1)}{2}}.
\end{eqnarray*}
Hence we obtain (1).

We prove (2) by the induction on $n$. 
We use the following notation
for short:
\[
(i,j,t,a)\seteq\langle k \rangle^{(i)}\lr{-k+2}{k}^{(j)}
\lr{-k}{k}^{\dv{t}}\lr{-k+2}{k-2}^{\dv{a}}\tvac.
\]
Then Lemma~\ref{lem:ijta}
implies that
\eqn
\pbw{-k}(i,j,t,a)&=&
[2t+2](i,j-1,t+1,a) \nn\\
&&+[j+1]q^{j-2t}(i,j+1,t,a-1)\nn \\
&&+[i+1]q^{2a-2t}(i+1,j,t,a).
\eneqn
Hence,  by assuming (2) for $n$, we have
\begin{eqnarray*}
&{}&[n+1]\pbw{-k}^{(n+1)}\pbw{-k+2,k-2}^{\dv{a}}\tvac=
\langle -k \rangle \langle -k \rangle^{(n)}\lr{-k+2}{k-2}^{\dv{a}}\tvac \\
&&=\quad\sum_{u=0}^{n}\sum_{i+j+2t=n,j+t=u}\left\{
\begin{array}{l}
[2t+2]q^{2ai+\frac{j(j-1)}{2}-i(t+u)}(i,j-1,t+1,a-u) \\
+[j+1]q^{2ai+\frac{j(j-1)}{2}-i(t+u)+j-2t}(i,j+1,t,a-u-1) \\
+[i+1]q^{2ai+\frac{j(j-1)}{2}-i(t+u)+2a-2u-2t}(i+1,j,t,a-u)
\end{array}
\right\}.
\end{eqnarray*}
Then in the right hand side, the coefficients of $(i',j',t',a-u')$ satisfying $i'+j'+2t'=n+1,j'+t'=u'$ are
\begin{eqnarray*}
&{}&[2t']q^{2ai'+\frac{(j'+1)j'}{2}-i'(t'-1+u')}+[j']q^{2ai'+\frac{(j'-1)(j'-2)}{2}-i'(t'+u'-1)+j'-1-2t'} \\
&{}&\hs{20ex}+[i']q^{2a(i'-1)+\frac{j'(j'-1)}{2}-(i'-1)(t'+u')+2a-2u'-2t'} \\
&&\qquad=q^{2ai'+\frac{j'(j'-1)}{2}-i'(t'+u')} 
\Bigl([2t']q^{j'+i'}+[j']q^{i'-2t'}+[i']q^{-(t'+u')}\Bigr)\\
&&\qquad=q^{2ai'+\frac{j'(j'-1)}{2}-i'(t'+u')} [n+1].
\end{eqnarray*}
We can prove (3) similarly as above. 
\end{proof}

\subsection{Actions of $E_k$, $F_k$ on the PBW basis}
For a $\theta$-restricted multisegment $\mb$, we set
$$\tP(\mb)=P_\theta(\mb)\tvac.$$
We understand $\tP(\mb)=0$ if $\mb$ is not a multisegment.

\begin{thm} \label{th:F-k}
For $k\in I_{>0}$ and 
a $\theta$-restricted multisegment 
$\mb=\sum_{-j\le i\le j}m_{i,j}\pbw{i,j}$, we have
\begin{eqnarray*}
&&F_{-k}\tP(\mb) \\
&&=\sum_{\ell>k}[m_{-k,\ell}+1]
q^{\sum\limits_{\ell'>\ell}(m_{-k+2,\ell'}-m_{-k,\ell'})}
\tP(\mb-\pbw{-k+2,\ell}+\pbw{-k,\ell}) \\
&&\ +q^{\sum\limits_{\ell>k}(m_{-k+2,\ell}-m_{-k,\ell})} 
[2m_{-k,k}+2]\tP(\mb-\pbw{-k+2,k}+\pbw{-k,k}) \\[1ex]
&&+q^{\sum\limits_{\ell>k}(m_{-k+2,k}-m_{-k,k})
+m_{-k+2,k}-2m_{-k,k}}[m_{-k+2,k}+1]
\tP(\mb-\delta_{k\not=1}\pbw{-k+2,k-2}+\pbw{-k+2,k}) \\[1ex]
&&+\kern-3ex\sum\limits_{-k+2<i\le k}\kern-2ex
q^{\sum\limits_{\ell>k}(m_{-k+2,k}-m_{-k,k})+
2m_{-k+2,k-2}-2m_{-k,k}+\kern-2ex\sum\limits_{-k+2<j<i}
(m_{j,k-2}-m_{j,k})}\\[-1ex]
&&\hs{40ex}\times[m_{i,k}+1]
\tP(\mb-\delta_{i<k}\pbw{i,k-2}+\pbw{i,k}). 
\end{eqnarray*}
\end{thm}
\begin{proof}
We divide $\mb$ into four parts,
$\mb=\mb_1+\mb_2+\mb_3+\delta_{k\not=1}m_{-k+2,k-2}\pbw{-k+2,k-2}$,
where
$\mb_1=\sum_{j>k}m_{i,j}\pbw{i,j}$,
$\mb_2=\sum_{j=k}m_{i,j}\pbw{i,j}$,
$\mb_3=\sum_{-k+2<i\le j\le k-2}m_{i,j}\pbw{i,j}$.
Then Proposition~\ref{pp1} implies
$$\tP(\mb)=\Pt(\mb_1)\Pt(\mb_2)\Pt(\mb_3)
\pbw{-k+2,k-2}^{\dv{m_{-k+2,k-2}}}\tvac.$$
If $k=1$, we understand $\pbw{-k+2,k-2}^{\dv{n}}=1$.
By Lemma~\ref{lem1} (1), we have
\eqn
&&\pbw{-k}\Pt(\mb_1)=
\sum_{\ell>k}q^{\sum_{\ell'>\ell}(m_{-k+2,\ell'}-m_{-k,\ell'})}
[m_{-k,\ell}+1]\Pt(\mb_1-\pbw{-k+2,\ell}+\pbw{-k,\ell})\\
&&\hs{20ex}+q^{\sum_{\ell>k}(m_{-k+2,\ell}-m_{-k,\ell})}
\Pt(\mb_1)\pbw{-k},
\eneqn
and  Lemma~\ref{lem1} (2) implies
\eqn
&&\pbw{-k}\Pt(\mb_2)=
[2m_{-k,k}+2]\Pt(\mb_2-\pbw{-k+2,k}+\pbw{-k,k})\\
&&\hs{20ex}+q^{m_{-k+2,k}-m_{-k,k}}
\Pt(\mb_2)\pbw{-k}.
\eneqn
Since we have
$\pbw{-k}\Pt(\mb_3)=\Pt(\mb_3)\pbw{-k}$,
we obtain
\eq
&&
\ba{ll}\pbw{-k}\tP(\mb)=&
\sum_{\ell>k}q^{\sum_{\ell'>\ell}(m_{-k+2,\ell'}-m_{-k,\ell'})}
[m_{-k,\ell}+1]\tP(\mb-\pbw{-k+2,\ell}+\pbw{-k,\ell})\\[1ex]
&+q^{\sum_{\ell>k}(m_{-k+2,\ell}-m_{-k,\ell})}[2m_{-k,k}+2]
\tP(\mb-\pbw{-k+2,k}+\pbw{-k,k})\\[1ex]
&+q^{\sum_{\ell\ge k}(m_{-k+2,\ell}-m_{-k,\ell})}
\Pt(\mb_1+\mb_2+\mb_3)\pbw{-k}
\pbw{-k+2,k-2}^{\dv{m_{-k+2,k-2}}}\tvac.
\ea
\label{eq:FtP1}
\eneq
By \eqref{eq:45}, we have
\eqn
\pbw{-k}\pbw{-k+2,k-2}^{\dv{m_{-k+2,k-2}}}\tvac
&=&\pbw{-k+2,k}\pbw{-k+2,k-2}^{\dv{m_{-k+2,k-2}-1}}\tvac\\
&&\hs{3ex}+\delta_{k\not=1}
q^{2m_{-k+2,k-2}}\pbw{k}\pbw{-k+2,k-2}^{\dv{m_{-k+2,k-2}}}\tvac.
\eneqn
Hence the last term in \eqref{eq:FtP1} is equal to
\eqn
&&q^{\sum_{\ell\ge k}(m_{-k+2,\ell}-m_{-k,\ell})-m_{-k,k}}
[m_{-k+2,k}+1]\tP(\mb-\delta_{k\not=1}\pbw{-k+2,k-2}+\pbw{-k+2,k})\\
&&
+\delta_{k\not=1}q^{\sum_{\ell\ge k}(m_{-k+2,\ell}-m_{-k,\ell})+2m_{-k+2,k-2}}
\Pt(\mb_1+\mb_2+\mb_3)\pbw{k}
\pbw{-k+2,k-2}^{\dv{m_{-k+2,k-2}}}\tvac.
\eneqn
For $k\not=1$,
Lemma~\ref{lem1} (4) implies
\eqn
&&\Pt(\mb_3)\pbw{k}
=\sum_{-k+2<i\le k}q^{\sum_{-k+2<j<i}m_{j,k-2}}
\pbw{i,k}\Pt(\mb_3-\delta_{i<k}\pbw{i,k-2}),
\eneqn
and Proposition~\ref{pp1} implies
\eqn
&&\Pt(\mb_2)\pbw{i,k}=
q^{-\sum_{j<i}m_{j,k}}[m_{i,k}+1]\Pt(\mb_2+\pbw{i,k}).
\eneqn
Hence we obtain 
\eqn
&&\Pt(\mb_1)\Pt(\mb_2)\Pt(\mb_3)\pbw{k}
\pbw{-k+2,k-2}^{\dv{m_{-k+2,k-2}}}\tvac\\
&&\hs{5ex}=
\sum_{-k+2<i\le k}q^{\sum_{-k+2<j<i}m_{j,k-2}-\sum_{-k\le j<i}m_{j,k}}
[m_{i,k}+1]\tP(\mb-\delta_{i<k}\pbw{i,k-2}+\pbw{i,k}).
\eneqn
Thus we obtain the desired result.
\QED

\begin{thm} \label{th:E-k}
For $k\in I_{>0}$ and 
a $\theta$-restricted multisegment 
$\mb=\kern-1ex\sum\limits_{-j\le i\le j}\kern-1ex m_{i,j}\pbw{i,j}$, we have
\eqn
&&E_{-k}\tP(\mb) \\
&&=(1-q^2)\sum_{\ell>k}q^{1+\sum\limits_{\ell'\ge\ell}
(m_{-k+2,\ell'}-m_{-k,\ell'})}
[m_{-k+2,\ell}+1]\tP(\mb-\pbw{-k,\ell}+\pbw{-k+2,\ell}) \\
&&
+(1-q^2)q^{1+\suml_{\ell>k}(m_{-k+2,\ell}-m_{-k,\ell})+m_{-k+2,k}-2m_{-k,k}}
[m_{-k+2,k}+1]\tP(\mb-\pbw{-k,k}+\pbw{-k+2,k})\\
&&+(1-q^2)\kern-2ex\sum_{-k+2<i\le k-2}\kern-2ex
q^{1+\suml_{\ell>k}(m_{-k+2,\ell}-m_{-k,\ell})+2m_{-k+2,k-2}-2m_{-k,k}
+\sum\limits_{-k+2<i'\le i}(m_{i,k-2}-m_{i'k})}\\[-2ex]
&&\hs{30ex}\times
[m_{i,k-2}+1]\tP(\mb-\pbw{i,k}+\pbw{i,k-2}) \\[1ex]
&&+\delta_{k\not=1}
(1-q^2)q^{1+\suml_{\ell>k}(m_{-k+2,\ell}-m_{-k,\ell})
+2m_{-k+2,k-2}-2m_{-k,k}}\\[-1ex]
&&\hs{20ex}\times
[2(m_{-k+2,k-2}+1)]\tP(\mb-\pbw{-k+2,k}+\pbw{-k+2,k-2})\\
&&+q^{\suml_{\ell>k}(m_{-k+2,\ell}-m_{-k,\ell})-2m_{-k,k}+\delta_{k\not=1}
{\displaystyle(}
1-m_{k,k}+2m_{-k+2,k-2}+\kern-2ex
\sum\limits_{-k+2<i\le k-2}(m_{i,k-2}-m_{i,k})
{\displaystyle)}}
\tP(\mb-\pbw{k}).
\eneqn
\end{thm}
\begin{proof}
We shall divide $\mb$ into
$\mb=\mb_1+\mb_2+\mb_3$ where
$\mb_1=\ssum\limits_{i\le j, j>k}m_{i,j}\pbw{i,j}$
and $\mb_2=\ssum_{i\le k}m_{i,k}\pbw{i,k}$
and $\mb_3=\ssum_{i\le j<k}m_{i,j}\pbw{i,j}$.
By \eqref{eq:Ea} and Proposition \ref{pp2}, we have
\eq
&&\ba{lcr}
E_{-k}\tP(\mb)&=&\Bigl(\bl(e'_{-k}P_\theta(\mb_1)\br)
P_\theta(\mb_2+\mb_3)+(\Ad(t_{-k})P_\theta(\mb_1))
(e'_{-k}P_\theta(\mb_2+\mb_3))\\[1ex]
&&+\Ad(t_{-k})\left\{
P_\theta(\mb_1)\bl(e_k^*\Pt(\mb_2)\br)
\Ad(t_k)\Pt(\mb_3)
\right\}\Bigr)\tvac.
\ea\label{eq:E-k}
\eneq
By Proposition~\ref{pp2}, the first term is 
\begin{eqnarray}
&&
\ba{l}
\bl(e'_{-k}P_\theta(\mb_1)\br)P_\theta(\mb_2+\mb_3)\\[1ex]
\quad=(1-q^2)\sum_{\ell>k}
q^{1+\sum\limits_{\ell'\ge\ell}(m_{-k+2,\ell'}-m_{-k,\ell'})}\\[1ex]
\hs{20ex}
\times[m_{-k+2,\ell}+1]P_\theta(\mb-\pbw{-k,\ell}+\pbw{-k+2,\ell}). 
\ea
\label{217eq1}
\end{eqnarray}
The second term is 
\begin{eqnarray}
&{}&(\Ad(t_{-k})P_\theta(\mb_1))(e'_{-k}P_\theta(\mb_2+\mb_3)) \nonumber \\
&&=q^{\sum_{\ell>k}(m_{-k+2,\ell}-m_{-k,\ell})}
\dfrac{[m_{-k,k}][m_{-k+2,k}+1]}{[2m_{-k,k}]}
(1-q^2)q^{1-m_{-k,k}+m_{-k+2,k}}\nn\\
&&\hs{20ex}\times P_\theta(\mb-\pbw{-k,k}+\pbw{-k+2,k}). 
\label{217eq2}
\end{eqnarray}
Let us calculate the last part of
\eqref{eq:E-k}. 
We have 
\eqn
&&\Ad(t_{-k})\Bigl(
P_\theta(\mb_1)\bl(e_k^*\Pt(\mb_2)\br)
\Ad(t_k)\Pt(\mb_3)\Bigr)\\
&&\qquad=
q^{\sum\limits_\ell(m_{-k+2,\ell}-m_{-k,\ell})
+\sum_{i\le k-2}m_{i,k-2}-\delta_{k=1}}
P_\theta(\mb_1)\bl(e_k^*\Pt(\mb_2)\br)
\Pt(\mb_3).
\eneqn
We have
\eqn
e_k^*\Pt(\mb_2)&=&
q^{1-m_k-\sum\limits_{i<k}m_{i,k}}
P_\theta(\mb_2-\pbw{k})\\
&&+(1-q^2)\sum\limits_{-k<i<k}
q^{1-m_{i,k}-\sum\limits_{i'<i}m_{i',k}}P_\theta(\mb_2-\pbw{i,k})
\pbw{i,k-2}\\
&&+\dfrac{[m_{-k,k}]}{[2m_{-k,k}]}(1-q^2)q^{1-m_{-k,k}}P(\mb_2-\pbw{-k,k})
\lr{-k}{k-2}.
\eneqn
For $-k<i<k$, we have
\eqn
&&\pbw{i,k-2}\Pt(\mb_3)
=q^{-\sum\limits_{i'>i}m_{i',k-2}}
[(1+\delta_{i=-k+2})(m_{i,k-2}+1)]\Pt(\mb_3+\pbw{i,k-2}).
\eneqn
By Lemma~\ref{lemq-1}, we have
\eqn
\pbw{-k,k-2}\Pt(\mb_3)\tvac
&=&q^{-\sum\limits_{-k+2\le k\le k-2}m_{i,k-2}}
\Pt(\mb_3)\pbw{-k,k-2}\tvac\\
&=&
q^{-\sum\limits_{-k+2\le k\le k-2}m_{i,k-2}-\delta_{k\not=1}}
\Pt(\mb_3)\pbw{-k+2,k}\tvac\\
&=&q^{-m_{-k+2,k-2}-\sum\limits_{-k+2\le i\le k-2}m_{i,k-2}-\delta_{k\not=1}}
\pbw{-k+2,k}\Pt(\mb_3)\tvac.
\eneqn
Hence we obtain
\eqn
&&\Pt(\mb_1)\bl(e_k^*\Pt(\mb_2)\br)\Pt(\mb_3)\tvac\\
&&=q^{1-\sum\limits_{i\le k}m_{i,k}}\tP(\mb-\pbw{k})\\
&&+(1-q^2)\sum\limits_{-k+2<i\le k-2}
q^{1-\sum\limits_{i'\le i}m_{i',k}-\sum\limits_{i'>i}m_{i',k-2}}\\
&&\hs{10ex}\times
[m_{i,k-2}+1]\tP(\mb-\pbw{i,k}+\pbw{i,k-2})\\[1ex]
&&+(1-q^2)
\delta_{k\not=1}
q^{1-m_{-k,k}-m_{-k+2,k}-\sum\limits_{-k+2<i}m_{i,k-2}}\\
&&\hs{10ex}\times
[2(m_{-k+2,k-2}+1)]\tP(\mb-\pbw{-k+2,k}+\pbw{-k+2,k-2})\\
&&+(1-q^2)q^{2(1-m_{-k,k})-m_{-k+2,k-2}-\sum\limits_{-k+2\le i\le k-2}m_{i,k-2}
-\delta_{k\not=1}}\\
&&\hs{10ex}\times
\dfrac{[m_{-k+2,k}+1][m_{-k,k}]}{[2m_{-k,k}]}P(\mb-\pbw{-k,k}+\pbw{-k+2,k}).
\eneqn
Hence the coefficient of $\tP(\mb-\pbw{k})$
in $E_{-k}\tP(\mb)$ is
\eqn
&&q^{\sum\limits_\ell(m_{-k+2,\ell}-m_{-k,\ell})
+\sum\limits_{i\le k-2}m_{i,k-2}-\delta_{k=1}+1
-\sum\limits_{i\le k}m_{i,k}}\\
&&=q^{\sum\limits_{\ell>k}(m_{-k+2,\ell}-m_{-k,\ell})
-2m_{-k,k}+\delta_{k\not=1}
{\displaystyle(}
1-m_{k,k}+2m_{-k+2,k-2}+\sum\limits_{-k+2<i\le k-2}(m_{i,k-2}-m_{i,k})
{\displaystyle)}}.
\eneqn
The coefficient of $\tP(\mb-\pbw{-k,k}+\pbw{-k+2,k})$
in $E_{-k}\tP(\mb)$ is
\eqn
&&(1-q^2)q^{1+\sum\limits_{\ell\ge k}(m_{-k+2,\ell}-m_{-k,\ell})}
\dfrac{[m_{-k,k}][m_{-k+2,k}+1]}{[2m_{-k,k}]}\\
&&+(1-q^2)q^{\sum\limits_\ell(m_{-k+2,\ell}-m_{-k,\ell})
+\sum\limits_{i\le k-2}m_{i,k-2}-\delta_{k=1}
+2(1-m_{-k,k})-m_{-k+2,k-2}-\sum\limits_{-k+2\le i\le k-2}m_{i,k-2}
-\delta_{k\not=1}}\\
&&\hs{10ex}\times
\dfrac{[m_{-k+2,k}+1][m_{-k,k}]}{[2m_{-k,k}]}\\
&&=
(1-q^2)q^{1+\sum_{\ell\ge k}(m_{-k+2,\ell}-m_{-k,\ell})}
\dfrac{[m_{-k,k}][m_{-k+2,k}+1]}{[2m_{-k,k}]}
(1+q^{-2m_{-k,k}})\\
&&=(1-q^2)q^{1-m_{-k,k}+\sum_{\ell\ge k}(m_{-k+2,\ell}-m_{-k,\ell})}
[m_{-k+2,k}+1]\\
&&=(1-q^2)q^{1+m_{-k+2,k}-2m_{-k,k}+\sum_{\ell>k}(m_{-k+2,\ell}-m_{-k,\ell})}
[m_{-k+2,k}+1].
\eneqn
For $-k+2<i\le k-2$,
the coefficient of $\tP(\mb-\pbw{i,k}+\pbw{i,k-2})$
in $E_{-k}\tP(\mb)$ is
\eqn
&&(1-q^2)q^{\sum\limits_\ell(m_{-k+2,\ell}-m_{-k,\ell})
+\sum\limits_{i'\le k-2}m_{i',k-2}-\delta_{k=1}
+1-\sum\limits_{i'\le i}m_{i',k}-\sum\limits_{i'>i}m_{i',k-2}}
[m_{i,k-2}+1]\\
&&=(1-q^2)q^{1+\sum\limits_{\ell>k}(m_{-k+2,\ell}-m_{-k,\ell})
+2m_{-k+2,k-2}-2m_{-k,k}+\sum\limits_{-k+2<i'\le i}(m_{i,k-2}-m_{i'k})}
[m_{i,k-2}+1].
\eneqn
Finally, for $k\not=1$,
the coefficient of $\tP(\mb-\pbw{-k+2,k}+\pbw{-k+2,k-2})$
in $E_{-k}\tP(\mb)$ is
\eqn
&&(1-q^2)q^{\sum\limits_\ell(m_{-k+2,\ell}-m_{-k,\ell})
+\sum\limits_{i\le k-2}m_{i,k-2}-\delta_{k=1}
+1-m_{-k,k}-m_{-k+2,k}-\sum\limits_{-k+2<i}m_{i,k-2}}
[2(m_{-k+2,k-2}+1)]\\
&&=(1-q^2)q^{1+\sum\limits_{\ell>k}(m_{-k+2,\ell}-m_{-k,\ell})
+2m_{-k+2,k-2}-2m_{-k,k}}
[2(m_{-k+2,k-2}+1)].
\eneqn

\QED

\begin{thm}\label{th:EF}
For $k>0$ and $\mb\in \Mt$, we have
\begin{eqnarray*}
E_k\tP(\mb)&=&
\ba[t]{r}\sum_{\ell>k}(1-q^2) 
q^{1+\sum_{\ell'\ge\ell}(m_{k+2,\ell'}-m_{k,\ell'})}
[m_{k+2,\ell}+1]\tP(\mb-\pbw{k,\ell}+\pbw{k+2,\ell})\\[1ex]
+q^{1+\sum_{\ell>k}(m_{k+2,\ell}-m_{k,\ell})-m_{k,k}}
\tP(\mb-\pbw{k}),
\ea \\
F_k\tP(\mb)&=&\sum_{\ell\ge k}
q^{\sum_{\ell'>\ell}(m_{k+2,\ell'}-m_{k,\ell'})}
[m_{k,\ell}+1]\tP(\mb-\delta_{\ell\not=k}\pbw{k+2,\ell}+\pbw{k,\ell}).
\end{eqnarray*}
\end{thm}
\begin{proof}
The first follows from $e^*_{-k}\Pt(\mb)=0$ and  Proposition~\ref{pp2},
and the second follows from Proposition~\ref{prop:div}.
\end{proof}

\section{Crystal basis of $\Vt$}
\subsection{Crystal structure on $\mathcal{M}_{\theta}$}
\label{subsec:KashMt}
We shall define the crystal structure on $\Mt$.
\begin{dfn}\label{def:crMt}
Suppose $k>0$. For a $\theta$-restricted multisegment 
$\mb=\suml_{-j\le i\le j}m_{i,j}\pbw{i,j}$, we set

\[
\varepsilon_{-k}(\mb)=\max\set{A_j^{(-k)}(\mb)}
{j\ge -k+2},
\]
where
\begin{eqnarray*}
A_j^{(-k)}(\mb)&=&\sum_{\ell\ge j}(m_{-k,\ell}-m_{-k+2,\ell+2})
\quad\text{for $j>k$,} \\
A_k^{(-k)}(\mb)&=&\sum_{\ell>k}(m_{-k,\ell}-m_{-k+2,\ell})
+2m_{-k,k}+\delta(\text{$m_{-k+2,k}$ is odd}), \\
A_j^{(-k)}(\mb)&=&\sum_{\ell>k}(m_{-k,\ell}-m_{-k+2,\ell})
+2m_{-k,k}-2m_{-k+2,k-2}+\kern-2ex\sum_{-k+2<i\le j+2}\kern-2ex m_{i,k}
-\kern-2ex\sum_{-k+2<i\le j}\kern-2ex m_{i,k-2} \\
&&\hs{50ex}\text{for $-k+2 \le j\le k-2$.}
\end{eqnarray*}
\bnum
\item
Let $n_f$ be the smallest
$\ell\ge-k+2$, with respect to the ordering
$\cdots> k+2>k>-k+2>\cdots>k-2$, such that $\eps_{-k}(\mb)=A_\ell^{(-k)}(\mb)$.
We define
\eqn
\tF_{-k}(\mb)&=&
\begin{cases}
\mb-\pbw{-k+2,n_f}+\pbw{-k,n_f}&\text{if $n_f>k$,}\\
\mb-\pbw{-k+2,k}+\pbw{-k,k}&\text{if $n_f=k$ and $m_{-k+2,k}$ is odd,}\\
\mb-\delta_{k\not=1}\pbw{-k+2,k-2}+\pbw{-k+2,k}
&\text{if $n_f=k$ and $m_{-k+2,k}$ is even,}\\
\mb-\delta_{n_f\not=k-2}\pbw{n_f+2,k-2}+\pbw{n_f+2,k}
&\text{if $-k+2\le n_f\le k-2$.}
\end{cases}\\
\eneqn

\item
If $\eps_{-k}(\mb)=0$, then $\tE_{-k}(\mb)=0$.
If $\eps_{-k}(\mb)>0$,  then
let $n_e$ be the largest $\ell\ge-k+2$, 
with respect to the above ordering, 
such that $\eps_{-k}(\mb)=A_\ell^{(-k)}(\mb)$.
We define
\eqn
\tE_{-k}(\mb)=
\begin{cases}
\mb-\pbw{-k,n_e}+\pbw{-k+2,n_e}&\text{if $n_e>k$,}\\
\mb-\pbw{-k,k}+\pbw{-k+2,k}
&\text{if $n_e=k$ and $m_{-k+2,k}$ is even,}\\
\mb-\pbw{-k+2,k}+\delta_{k\not=1}\pbw{-k+2,k-2}
&\text{if $n_e=k$ and $m_{-k+2,k}$ is odd,}\\
\mb-\pbw{n_e+2,k}+\delta_{n_e\not=k-2}\pbw{n_e+2,k-2}
&\text{if $-k+2\le n_e\le k-2$.}
\end{cases}&&
\eneqn
\ee
\end{dfn}


\begin{rem}
For $0<k \in I$, 
the actions of $\widetilde{E}_{-k}$ and $\widetilde{F}_{-k}$ on 
$\mb\in\mathcal{M}_\theta$ are described by the following algorithm.
\be[{\rm Step 1.}]
\item
Arrange segments in $\mb$ of the form
$\pbw{-k,j}$ $(j>k)$, $\pbw{-k+2,j}$ $(j>k)$,
$\pbw{i,k}$ $(-k\le i\le $k), $\pbw{i,k-2}$ $(-k+2\le i\le k-2)$
in the order
\eqn
&&\cdots,\pbw{-k,k+2},\pbw{-k+2,k+2},\;\pbw{-k,k},
\pbw{-k+2,k},\pbw{-k+2,k-2},\\
&&\hs{10ex}\pbw{-k+4,k},\pbw{-k+4,k-2},\cdots,
\pbw{k-2,k},\pbw{k-2,k-2},\pbw{k}.
\eneqn

\item Write signatures for each segment contained in 
$\mb$ by the following rules.
\bnum
\item If a segment is not $\lr{-k+2}{k}$, then
\begin{itemize}
\item{For $\lr{-k}{k}$, write $--$,}
\item{For $\lr{-k}{j}$ with $j> k$, write $-$,} 
\item{For $\lr{-k+2}{k-2}$ with $k>1$, write $++$,} 
\item{For $\lr{-k+2}{j}$ with $j>k$, write $+$,}
\item{For $\lr{j}{k}$ with $-k+2<j\le k$, write $-$,} 
\item{For $\lr{j}{k-2}$ with $-k+2<j\le k-2$, write $+$,}
\item{Otherwise, write no signature.} 
\end{itemize}
\item
For segments $m_{-k+2,k}\lr{-k+2}{k}$, 
if $m_{-k+2,k}$ is even, then write no signature, 
and if $m_{-k+2,k}$ is odd, then write $-+$. 
\ee
\item
In the resulting sequence of $+$ and $-$, 
delete a subsequence of the form $+-$ 
and keep on deleting until no such subsequence remains. 
\ee
Then we obtain a sequence of the form $-- \cdots -++ \cdots +$. 
\be[{\rm(1)}]
\item
$\varepsilon_{-k}(\mb)$ 
is the total number of $-$ in the resulting sequence. 
\item $\widetilde{F}_{-k}(\mb)$ is given as follows: 
\bnum
\item
if the leftmost $+$ corresponds to a segment $\lr{-k+2}{j}$ for $j >k$, 
then replace it with $\lr{-k}{j}$,
\item
if the leftmost $+$ corresponds to a segment 
$\lr{j}{k-2}$ for $-k+2\le j\le k-2$, 
then replace it with $\lr{j}{k}$,
\item
if the leftmost $+$ corresponds to segment $m_{-k+2,k}\lr{-k+2}{k}$, 
then replace one of the segments with $\lr{-k}{k}$,
\item if no $+$ exists, add a segment $\lr{k}{k}$ to $\mb$.
\ee
\item
$\widetilde{E}_{-k}(\mb)$ is given as follows:
\bnum
\item
if the rightmost $-$ corresponds to a segment $\lr{-k}{j}$ for $j\ge k$, 
then replace it with $\lr{-k+2}{j}$,
\item
if the rightmost $-$ corresponds to a segment $\lr{j}{k}$
for $-k+2<j<k$, 
then replace it with $\lr{j}{k-2}$,
\item
if the rightmost $-$ corresponds to segments $m_{-k+2,k}\lr{-k+2}{k}$, 
then replace one of the segment with $\lr{-k+2}{k-2}$,
\item
if the rightmost $-$ corresponds to a segment $\lr{k}{k}$ for $k>1$,
then delete it,
\item
if no $-$ exists, then $\widetilde{E}_{-k}(\mb)=0$. 
\ee
\ee
\end{rem}

\begin{exa}
\begin{enumerate}
\item We shall write $\{a,b\}$ for $a\lr{-1}{1}+b\langle 1 \rangle$. The following diagram is the part of the crystal graph of $\Bz$ that concerns only the $1$-arrows and the $(-1)$-arrows.
$$\xymatrix@R=.1em@C=2em{&&&&\{0,4\}\ar@<.2pc>[r]^{1}\ar@<-.2pc>[r]_{-1}&
\{0,5\}\cdots\\
&&\{0,2\}\ar@<.2pc>[r]^{1}\ar@<-.2pc>[r]_{-1}&\{0,3\}
\ar[ru]^{1}\ar[rd]|{-1}\\
\vac\ar@<.2pc>[r]^(.4){1}\ar@<-.2pc>[r]_(.35){-1}&
\{0,1\}\ar[ru]^{1}\ar[rd]_{-1}&&&\{1,2\}\ar@<.2pc>[r]^{1}\ar@<-.2pc>[r]_{-1}&\{1,3\}\cdots\\
&&\{1,0\}\ar@<.2pc>[r]^{1}\ar@<-.2pc>[r]_{-1}&\{1,1\}\ar[ru]|{1}\ar[rd]_{-1}\\
&&&&\{2,0\}\ar@<.2pc>[r]^{1}\ar@<-.2pc>[r]_{-1}&{\{2,1\}\cdots}
}$$
Especially the part of $(-1)$-arrows is the following diagram.
$$\xymatrix@C=6ex{
\{0,2n\}\ar[r]^(.45){-1}&
\{0,2n+1\}\ar[r]^(.55){-1}&
\{1,2n\}\ar[r]^(.45){-1}&
\{1,2n+1\}\ar[r]^(.45){-1}&
\{2,2n\}\cdots\cdots
}$$
\item The following diagram is the part of the crystal graph of 
$\CB_\theta(0)$ that concerns only the $(-1)$-arrows and the $(-3)$-arrows. 
This diagram is, as a graph, isomorphic to the crystal graph 
of $A_2$.

$$
\xymatrix@R=.1em@C=2em{&&&&2\lr{-1}{1}\\
&&&\lr{-1}{1}+\langle 1 \rangle \ar[ru]^{-1}\ar[rd]_{-3}\\
&& \lr{-1}{1}\ar[ru]^{-1}\ar[rd]_{-3} && \lr{-1}{3}+\langle 1 \rangle \\
& \langle 1 \rangle\ar[ru]^{-1}\ar[rdd]_{-3} 
&& \lr{-1}{3}\ar[ru]_{-1}\ar[rd]_{-3} & \\
&&&& \lr{-3}{3} \\
&& \lr{1}{3}\ar[ruu]^{-1}\ar[rdddd]_{-3} & \\
&&&& \pbw{3}+\pbw{-1,1}+\pbw{1} \\
&&& \langle 3 \rangle+\lr{-1}{1}\ar[ru]^(.4){-1}\ar[rd]_{-3} & \\
\vac\ar[ruuuuu]^{-1}\ar[rddddd]^{-3} &&&& \lr{-1}{3}+\langle 3 \rangle \\
&&& \pbw{1,3}+\langle 3 \rangle\ar[ru]_{-1}\ar[rd]_{-3} & \\
&&&& \pbw{1,3}+2\pbw{3}\\
&& \langle 3 \rangle+\langle 1 \rangle\ar[ruuuu]^{-1}\ar[rdd]_{-3} & \\
&&&& 2\langle 3\rangle+\pbw{-1,1}\\
& \langle 3 \rangle\ar[ruu]^{-1}\ar[rd]_{-3} 
&& 2\langle 3 \rangle+\langle 1 \rangle\ar[ru]_{-1}\ar[rd]_{-3} & \\
&& 2\langle 3 \rangle \ar[ru]^{-1}\ar[rd]_{-3} && 3\pbw{3}+\pbw{1}\\
&&&3\langle 3 \rangle \ar[ru]_{-1}\ar[rd]_{-3}\\
&&&&4\langle 3 \rangle
}$$
\item Here is the part of the crystal graph of $\Bz$
that concerns only the $n$-arrows and the $(-n)$-arrows
for an odd integer $n\ge 3$:
$$\xymatrix@C=6ex{
\vac\ar@<.2pc>[r]^(.5){n}\ar@<-.2pc>[r]_(.45){-n}&
\pbw{n}\ar@<.2pc>[r]^(.45){n}\ar@<-.2pc>[r]_(.4){-n}&
2\pbw{n}\ar@<.2pc>[r]^(.45){n}\ar@<-.2pc>[r]_(.4){-n}&
3\pbw{n}\ar@<.2pc>[r]^(.45){n}\ar@<-.2pc>[r]_(.4){-n}&
\cdots\cdots
}$$
\end{enumerate}
\end{exa}

\Lemma
For $k\in I_{>0}$, the data $\tE_{-k}$,
$\tF_{-k}$, $\eps_{-k}$
define a crystal structure on $\Mt$,
namely we have
\bnum
\item
$\tF_{-k}\Mt\subset \Mt$ and
$\tE_{-k} \Mt\subset \Mt\sqcup\{0\}$,
\item
$\tF_{-k}\tE_{-k}(\mb)=\mb$ if $\tE_{-k}(\mb)\not=0$, 
and $\tE_{-k}\circ\tF_{-k}=\id$,
\item 
$\eps_{-k}(\mb)=\max\set{n\ge0}{\tE_{-k}^n(\mb)\not=0}$ for any 
$\mb\in\Mt$.
\enum
\enlemma
\Proof
We shall first show that, for 
$\mb=\sum_{-j\le i\le j}m_{i,j}\pbw{i,j}\in \Mt$, 
$\tF_{-k}(\mb)$ is $\theta$-restricted,
$\tE_{-k}\tF_{-k}(\mb)=\mb$ and 
$\eps_{-k}(\tF_{-k}\mb)=\eps_{-k}(\mb)+1$.
Let $A_j\seteq A_j^{(-k)}(\mb)$ ($j\ge -k+2$) 
and let $n_f$ be as in Definition~\ref{def:crMt}.
Set $\mb'=\tF_{-k}\mb$.
Let $A'_j=A_j^{(-k)}(\mb')$
and let $n_e'$ be $n_e$ for $\mb'$.
\bnum
\item Assume $n_f>k$.
Since $A_{n_f}>A_{n_f-2}=A_{n_f}+m_{-k,n_f-2}-m_{-k+2,n_f}$,
we have $m_{-k,n_f-2}<m_{-k+2,n_f}$.
Hence $\mb'=\mb-\pbw{-k+2,n_f}+\pbw{-k,n_f}$ is $\theta$-restricted.
Then we have
$$A'_j=\begin{cases}A_j & \text{if $j>n_f$,}\\
A_j+1 & \text{if $j=n_f$,}\\
A_j+2 & \text{if $j<n_f$.}
\end{cases}
$$
Hence $\eps_{-k}(\mb')=A_{n_f}+1=\eps_{-k}(\mb)+1$
and $n'_e=n_f$, which implies $\mb=\tE_{-k}(\mb')$.

\item Assume $n_f=k$.
\be[(a)]
\item
If $m_{-k+2,k}$ is odd, then
$\mb'=\mb-\pbw{-k+2,k}+\pbw{-k,k}$ is $\theta$-restricted.
We have
$$A'_j=\begin{cases}A_j & \text{if $j>k$,}\\
A_j+1 & \text{if $j=k$,}\\
A_j+2 & \text{if $j<k$,}
\end{cases}
$$
Hence $\eps_{-k}(\mb')=\eps_{-k}(\mb)+1$
and $n'_e=k$, which implies $\mb=\tE_{-k}(\mb')$.
\item
Assume that $m_{-k+2,k}$ is even.
If $k\not=1$, then
$A_k>A_{-k+2}=A_k-2m_{-k+2,k-2}$, and hence $m_{-k+2,k-2}>0$.
Therefore
$\mb'=\mb-\delta_{k\not=1}\pbw{-k+2,k-2}+\pbw{-k+2,k}$ is $\theta$-restricted.
We have
$$A'_j=\begin{cases}A_j & \text{if $j>k$,}\\
A_j+1 & \text{if $j=k$,}\\
A_j+2 & \text{if $j<k$.}
\end{cases}
$$
Hence $\eps_{-k}(\mb')=\eps_{-k}(\mb)+1$
and $n'_e=k$, which implies $\mb=\tE_{-k}(\mb')$.
\ee
\item Assume $-k+2\le  n_f< k-2$.
Since $A_{n_f}>A_{n_f+2}=A_{n_f}+m_{n_f+4,k}-m_{n_f+2,k-2}$,
we have $m_{n_f+2,k-2}>m_{n_f+4,k}$.
Hence $\mb'=\mb-\pbw{n_f+2,k-2}+\pbw{n_f+2,k}$ is $\theta$-restricted.
Then we have
$$A'_j=\begin{cases}A_j & \text{if $j>n_f$,}\\
A_j+1 & \text{if $j=n_f$,}\\
A_j+2 & \text{if $j<n_f$.}
\end{cases}
$$
(Here the ordering is as in Definition~\ref{def:crMt} (i).)
Hence $\eps_{-k}(\mb')=\eps_{-k}(\mb)+1$
and $n'_e=n_f$, which implies $\mb=\tE_{-k}\mb'$.
\item Assume $n_f=k-2$.
It is obvious that $\mb'=\mb+\pbw{k}$ is $\theta$-restricted.
We have
$$A'_j=\begin{cases}A_j & \text{if $j\not=n_f$,}\\
A_j+1 & \text{if $j=n_f$.}
\end{cases}
$$
Hence $\eps_{-k}(\mb')=\eps_{-k}(\mb)+1$
and $n'_e=n_f$, which implies $\mb=\tE_{-k}(\mb')$.
\ee
Similarly, we can prove that
if $\eps_{-k}(\mb)>0$, then
$\tE_{-k}(\mb)$ is $\theta$-restricted and
$\tF_{-k}\tE_{-k}(\mb)=\mb$.
Hence we obtain the desired results.
\QED

\begin{dfn}
For $k\in I_{>0}$, we define
$\tF_{k}$, $\tE_k$ and $\eps_k$ by the same rule as in 
{\rm Definition~\ref{defKop}} for
$\tf_k$, $\te_k$ and $\eps_k$.
\end{dfn}

Since it is well-known that it gives a crystal structure on $\M$,
we obtain the following result.

\begin{thm}\label{th:Excr}
By $\tF_k$, $\tE_k$, $\eps_k$ \ro $k\in I$\rf,
$\Mt$ is a crystal, 
namely, we have 
\bnum
\item
$\tF_{k}\Mt\subset \Mt$ and
$\tE_{k} \Mt\subset \Mt\sqcup\{0\}$,
\item
$\tF_{k}\tE_{k}(\mb)=\mb$ if $\tE_{k}(\mb)\not=0$, 
and $\tE_{k}\circ\tF_{k}=\id$,
\item 
$\eps_{k}(\mb)=\max\set{n\ge0}{\tE_k^n(\mb)\not=0}$ for any 
$\mb\in\Mt$.
\enum
\end{thm}

The crystal $\Mt$ has a unique highest weight vector.
\Lemma\label{lem:ht}
 If $\mb\in\Mt$ satisfies that
$\eps_{k}(\mb)=0$ for any $k\in I$, then
$\mb=\emptyset$.
Here $\emptyset$ is the empty multisegment.
In particular, for any $\mb\in\Mt$, there exist $\ell\ge0$ and
$i_1,\ldots,i_\ell\in I$ such that $\mb=\tF_{i_1}\cdots\tF_{i_\ell}\emptyset$.
\enlemma
\Proof
Assume $\mb\not=\emptyset$.
Let $k$ be the largest $k$ such that
$m_{k,j}\not=0$ for some $j$.
Then take the largest $j$ such that
$m_{k,j}\not=0$.
Then $j\ge|k|$. Moreover, we have
$m_{k+2,\ell}=0$ for any $\ell$, and
$m_{k,\ell} =0$ for any $\ell>j$. Hence we have
$$A_{j}^{(k)}(\mb)=
\begin{cases}
2m_{k,j}&\text{if $k=-j$,}\\
m_{k,j}&\text{otherwise.}
\end{cases}
$$
Hence $\eps_k(\mb)\ge A_{j}^{(k)}(\mb)>0$.
\QED
\subsection{A criterion for crystals}

We shall give a criterion for a basis to be
a crystal basis. Although we treat the
case for modules over $\mathcal{B}(\g)$
in this paper, 
similar results hold also for $\U$.

Let $\K[e,f]$ be the ring generated by $e$ and $f$
with the defining relation $ef=q^{-2}fe+1$.
We define the divided power by $f^{(n)}=f^n/[n]!$.

Let $P$ be a free $\Z$-module, and let $\alpha$ be a
non-zero element of $P$.

Let $M$ be a $\K[e,f]$-module.
Assume that $M$ has a weight decomposition
$M=\oplus_{\xi\in P}M_\xi$,
and $eM_{\la}\subset M_{\la+\al}$
and $fM_{\la}\subset M_{\la-\al}$.

Assume the following finiteness conditions:
\eq
&&\text{for any $\la\in P$, 
$\dim M_\la<\infty$ and $M_{\la+n\al}=0$ for $n\gg0$.}
\eneq
Hence for any $u\in M$,
we can write $u=\sum_{n\ge0}f^{(n)}u_n$ with $eu_n=0$.
We define endomorphisms $\te$ and $\tf$ of $M$ by
\eqn
\te u=\sum_{n\ge1}f^{(n-1)}u_n,\\
\tf u=\sum_{n\ge0}f^{(n+1)}u_n.
\eneqn
Let $B$ be a crystal with weight decomposition by $P$.
In this paper, we consider only the following type of crystals.
We have $\wt\cl B\to P$, $\tf\cl B\to B$,
$\te\cl B\to B\sqcup\{0\}$, $\eps\cl B\to\Z_{\ge0}$
satisfying the following properties, where $B_\la\seteq\wt^{-1}(\la)$:
\bnum
\item
$\tf B_\la\subset B_{\la-\al}$ and
$\te B_\la\subset B_{\la+\al}\sqcup\{0\}$
for any $\la\in P$,
\item
$\tf\te(b)=b$ if $\te b\not=0$, and $\te\circ\tf=\id_B$,
\item 
for any $\la\in P$, $B_\la$ is a finite set and
$B_{\la+n\al}=\emptyset$ for $n\gg0$,
\item
$\eps(b)=\max\set{n\ge0}{\te^n b\not=0}$ for any $b\in B$.
\enum

Set $\ord(a)=\sup\set{n\in\Z}{a\in q^n\A_0}$
for $a\in \K$.
We understand $\ord(0)=\infty$.

\smallskip
Let $\{C(b)\}_{b\in B}$ be a system of generators of 
$M$ with $C(b)\in M_{\wt(b)}$: $M=\sum_{b\in B}\K C(b)$.

Let $\xi$ be a map from $B$ to an ordered set.
Let $c\cl \Z\to\R$,
$f\cl \Z\to \R$ and
$e\cl \Z\to \R$.
Assume that a decomposition $B=B'\cup B''$ is given.

Assume that we have expressions:
\eq
eC(b)=\sum_{b'\in B}E_{b,b'}C(b'),\\
fC(b)=\sum_{b'\in B}F_{b,b'}C(b').
\eneq

Now consider the following conditions
for these data, where $l=\eps(b)$ and $l'=\eps(b')$:
{\allowdisplaybreaks
\eq
&&\text{$c(0)=0$, and $c(n)>0$ for $n\not=0$,}\label{eq:8}\\
&&c(n)\le n+c(m+n)+e(m)\quad\text{for $n\ge0$,}\label{eq:9}\\
&&c(n)\le c(m+n)+f(m)\quad\text{for $n\le0$,}\label{eq:10}\\
&&c(n)+f(n)>0\quad\text{for $n>0$,}\label{eq:13}\\
&&c(n)+e(n)>0\quad\text{for $n>0$,}\label{eq:14}\\
&&\ord(F_{b,b'})\ge -\ell+f(\ell+1-\ell'),\label{eq:4}\\
&&\ord(E_{b,b'})\ge 1-\ell+e(\ell-1-\ell'),\label{eq:5}\\
&&F_{b,\tf b}\in q^{-\ell}(1+q\A_0),\label{eq:6}\\
&&E_{b,\te b}\in q^{1-\ell}(1+q\A_0)\quad\text{if $\ell>0$,}
\label{eq:7}\\
&&\ord(F_{b,b'})>-\ell+f(\ell+1-\ell')\quad
\text{if  $b'\not=\tf b$, $\xi(\tf b)\not>\xi(b')$,}\label{eq:11.5}\\
&&\text{$\ord(F_{b,b'})>-\ell+f(\ell+1-\ell')$
if $\tf b\in B'$, $b'\not=\tf b$ and $\ell\le \ell'-1$,}\label{eq:15}\\
&&\text{$\ord(E_{b,b'})>1-\ell+e(\ell-1-\ell')$
if $b\in B''$, $b'\not=\te b$ and $\ell\le \ell'+1$.}\label{eq:16}
\eneq
}

\begin{thm}\label{th:crcr}
Assume the conditions \eqref{eq:8}--\eqref{eq:16}.
Set $L=\sum_{b\in B}\A_0C(b)$.
Then we have
$\te L\subset L$ and $\tf L\subset L$.
Moreover we have
$$\te C(b)\equiv C(\te b)\mod qL\quad\text{and}\quad
\tf C(b)\equiv C(\tf b) \mod qL\quad\text{for any $b\in B$.}$$
Here we understand $C(0)=0$.
\end{thm}

We shall divide the proof into several steps.

Write
$$C(b)=\sum_{n\ge0}f^{(n)}C_n(b)\quad\text{with $eC_n(b)=0$.}$$
Set
$$L_0=\sum_{b\in B,\;n\ge0}\A_0 f^{(n)}C_0(b).$$

Set for $u\in M$,
$\ord(u)=\sup\set{n\in\Z}{u\in q^nL_0}$.
If $u=0$ we set $\ord(u)=\infty$,
and if $u\not\in\cup_{n\in\Z}q^nL_0$,
then $\ord(u)=-\infty$.

We shall use the following two recursion formulas
\eqref{eq:E} and \eqref{eq:F}.

We have
\eqn
eC(b)&=&\sum_{n\ge1}q^{1-n}f^{(n-1)}C_n(b)\\
&=&\sum_{n\ge0}E_{b,b'}f^{(n)}C_n(b').
\eneqn
Hence we have
\eq
C_n(b)=\sum_{b'\in B_{\la+\al}
}q^{n-1}E_{b,b'}C_{n-1}(b')\quad\text{for $n>0$ and $b\in B_\la$.}
\label{eq:E}
\eneq

If $\ell\seteq\eps(b)>0$, then we have
\eqn
fC(\te b)&=&\sum_{b'\in B,\; n\ge0}F_{\te b,b'}f^{(n)}C_n(b')\\
&=&\sum_{n\ge0}[n+1]f^{(n+1)}C_{n}(\te b).
\eneqn
Hence, we have by \eqref{eq:6}
\eqn
\ba{rcl}
\delta_{n\not=0}[n]C_{n-1}(\te b)&=&\sum_{b'}F_{\te b,b'}C_n(b')\\
&\in& q^{1-\ell}(1+q\A_0)C_n(b)+\sum_{b'\not=b}F_{\te b,b'}C_n(b').
\ea
\eneqn
Therefore we obtain
\eq
&&C_n(b)\in \delta_{n\not=0}(1+q\A_0)q^{\ell-n}C_{n-1}(\te b)
+\sum_{b'\not=b}q^{\ell-1}\A_0F_{\te b,b'}C_n(b')\quad\text{if $\ell>0$.}
\label{eq:F}
\eneq

\Lemma
$\ord(C_n(b))\ge c(n-\ell)$ for any $n\in\Z_{\ge0}$ and $b\in B$,
where $\ell\seteq\eps(b)$.
\enlemma
\Proof
For $\la\in P$, 
we shall show the assertion for $b\in B_\la$
by the induction on $\sup\set{n\in\Z}{M_{\la+n\al}\not=0}$.
Hence we may assume
\eq
&&\text{
$\ord(C_n(b))\ge c(n-\ell)$ for any $n\in\Z_{\ge0}$ and $b\in B_{\la+\al}$.}
\label{eq:indh}
\eneq

\noindent
(i) Let us first show $C_n(b)\in\K L_0$.

Since it is trivial for $n=0$, assume that $n>0$.
Since $C_{n-1}(b')\in \K L_0$ for $b'\in B_{\la+\al}$
by the induction assumption \eqref{eq:indh}, we have
$C_{n}(b)\in \K L_0$ by \eqref{eq:E}.

\noindent
(ii) Let us show that $\ord(C_n(b))\ge c(n-\ell)$ for $n\ge\ell$.

If $n=0$, then $\ell=0$ and the assertion is trivial by \eqref{eq:8}.
Hence we may assume that $n>0$.

We shall use \eqref{eq:E}. For $b'\in B_{\la+\al}$, we have
$$\ord(C_{n-1}(b'))\ge c(n-1-\ell')\quad\text{where $\ell'=\eps(b')$}$$
by the induction hypothesis \eqref{eq:indh}.
On the other hand,
$\ord(E_{b,b'})\ge 1-\ell+e(\ell-1-\ell')$ by \eqref{eq:5}.
Hence, 
\eqn
\ord(q^{n-1}E_{b,b'}C_{n-1}(b'))
&\ge&(n-1)+\bl(1-\ell+e(\ell-1-\ell')\br)+c(n-1-\ell')\\
&=&(n-\ell)+e(\ell-1-\ell')+c((n-\ell)+(\ell-1-\ell'))\\
&\ge&c(n-\ell)
\eneqn
by \eqref{eq:9}.

\medskip
\noindent
(iii)\ In the general case, let us set
 $r=\min\set{\ord(C_n(b))-c(n-\eps(b))}
{\text{$b\in B_\la$, $n\ge0$}}\in\R\cup\{\infty\}$.
Assuming $r<0$, 
we shall prove
$$\ord(C_n(b))>c(n-\ell)+r\quad\text{for any $b\in B_\la$,}$$
which leads a contradiction.

By the induction on $\xi(b)$,
we may assume that
\eq&&\text{if $\xi(b')<\xi(b)$, then
$\ord(C_n(b'))>c(n-\ell')+r$ where $\ell'\seteq\eps(b')$.}
\label{ind.hyp}
\eneq

By (ii), we may assume that $n<\ell$. Hence $\te b\in B$.
By the induction hypothesis \eqref{eq:indh}, we have
$\ord(q^{\ell-n}C_{n-1}(\te b))\ge \ell-n+c((n-1)-(\ell-1))
\ge c(n-\ell)>c(n-\ell)+r$.
By \eqref{eq:F}, it is enough to show
$$\ord(q^{\ell-1}F_{\te b,b'}C_n(b'))>c(n-\ell)+r\quad\text{for $b'\not=b$.}$$
We shall divide its proof into two cases.

\be[(a)]
\item $\xi(b')<\xi(b)$.

In this case, \eqref{ind.hyp} implies
$\ord(C_n(b'))>c(n-\ell')+r$.
Hence 
\eqn
\ord(q^{\ell-1}F_{\te b,b'}C_n(b'))
&>&(\ell-1)+(1-\ell+f(\ell-\ell'))+c(n-\ell')+r\\
&=&f(\ell-\ell')+c((n-\ell)+(\ell-\ell'))+r\ge c(n-\ell)+r
\eneqn
by \eqref{eq:4} and \eqref{eq:10}.

\item Case $\xi(b')\not<\xi(b)$.

In this case, $\ord(F_{\te b,b'})>1-\ell+f(\ell-\ell')$ by \eqref{eq:11.5},
and
$\ord(C_n(b'))\ge c(n-\ell')+r$.
Hence, 
\eqn
\ord(q^{\ell-1}F_{\te b,b'}C_n(b'))
&>&(\ell-1)+(1-\ell+f(\ell-\ell'))+c(n-\ell')+r\\
&=&f(\ell-\ell')+c((n-\ell)+(\ell-\ell'))+r\ge c(n-\ell)+r.
\eneqn
\ee
\QED

\Lemma
$\ord(C_\ell(b)-C_{\ell-1}(\te b))>0$ for $\ell\seteq\eps(b)>0$.
\enlemma
\Proof

We divide the proof into two cases: $b\in B'$ and $b\in B''$.

\bnum
\item $b\in B'$.

By \eqref{eq:F},
it is enough to show
$$\ord(q^{\ell-1}F_{\te b,b'}C_\ell(b'))>0\quad\text{for $b'\not=b$.}$$

\be[(a)]
\item Case $\ell>\ell'\seteq\eps(b')$.

We have
\eqn
\ord(q^{\ell-1}F_{\te b,b'}C_\ell(b'))
&\ge&(\ell-1)+(1-\ell+f(\ell-\ell'))+c(\ell-\ell')>0
\eneqn
by \eqref{eq:13}.

\item Case $\ell\le\ell'$.

We have $\ord(F_{\te b,b'})>1-\ell+f(\ell-\ell')$ by \eqref{eq:15}.
Hence
\eqn
\ord(q^{\ell-1}F_{\te b,b'}C_\ell(b'))
&>&(\ell-1)+(1-\ell+f(\ell-\ell'))+c(\ell-\ell')\ge0
\eneqn
by \eqref{eq:10} with $n=0$.

%
\ee

\item Case $b\in B''$.

We use \eqref{eq:E}.
By \eqref{eq:7}, it is enough to show that
$$\text{$\ord(q^{\ell-1}E_{b,b'}C_{\ell-1}(b'))>0$ for $b'\not=\te b$.}$$

\be[(a)]
\item Case $\ell-1>\ell'$.

$
\ord(q^{\ell-1}E_{b,b'}C_{\ell-1}(b'))
\ge e(\ell-1-\ell')+c(\ell-1-\ell')>0$ by \eqref{eq:5} and \eqref{eq:14}.

\item Case $\ell-1\le\ell'$.

 $\ord(E_{b,b'})>1-\ell+e(\ell-1-\ell')$ by \eqref{eq:16}, and
$
\ord(q^{\ell-1}E_{b,b'}C_{\ell-1}(b'))
>e(\ell-1-\ell')+c(\ell-1-\ell')\ge0$
by \eqref{eq:9} with $n=0$.

%

\ee
\enum
\QED
Hence we have
\eqn
&&C_{n}(b)\equiv 0 \mod qL_0\quad \text{for $n\not=\ell\seteq\eps(b)$,}\\
&&C_{\ell}(b)\equiv C_0(\te^\ell b) \mod qL_0,\\
&&C(b)\equiv f^{(\ell)}C_\ell(b) \mod qL_0,\\
&&\tf C(b)\equiv C(\tf b) \mod qL_0,\\
&&\te C(b)\equiv C(\te b) \mod qL_0,\\
&&L_0\seteq\sum_{b\in B,\;n\ge0}\A_0f^{(n)}C_0(b)=\sum_{b\in B}\A_0C(b).
\eneqn
Indeed, the last equality follows from the fact that
$\{C(b)\}_{b\in B}$ generates $L_0/qL_0$.

Thus we have completed the proof of Theorem~\ref{th:crcr}.

The following is the special case where
$B'=B''=B$ and $\xi(b)=\eps(b)$.
\begin{cor}\label{cor:crcr}
Assume \eqref{eq:8}--\eqref{eq:7} and
\eq
&&\ord(F_{b,b'})>-\ell+f(1+\ell-\ell')\quad
\text{if $\ell<\ell'$ and $b'\not=\tf b$, }\label{eq:43}\\
&&\ord(E_{b,b'})>1-\ell+e(\ell-1-\ell')
\quad\text{if  $\ell\le\ell'+1$ and $b'\not=\te b$.}\label{eq:44}
\eneq
Then the assertions of\/ {\rm Theorem~\ref{th:crcr}} hold.
\end{cor}

\subsection{Estimates of the order of coefficients}
By applying Theorem~\ref{th:crcr}, we shall show that
$\{\Pt(\mb)\vac\}_{\mb\in\Mt}$ is a crystal basis of $\Vt$ and
its crystal structure coincides with the one given in 
\S\;\ref{subsec:KashMt}.

We define
$c, f, e\cl\Z\to\Q$ by $c(n)=\vert n/2\vert$ and
$f(n)=e(n)=n/2$. 
Then the conditions \eqref{eq:8}--\eqref{eq:14} are obvious.
Set $\xi(\mb)=(-1)^{m_{-k+2,k}}m_{-k,k}$ and
\eqn
B''&=&\set{\mb\in\Mt}{-k+2\le n_e(\mb)<k}
\cup\set{\mb\in\Mt}{\text{$m_{-k+2,k}(\mb)$ is odd}},\\
B'&=&\Mt\setminus B''.
\eneqn
Here $n_e(\mb)$ is $n_e$ given in Definition~\ref{def:crMt} (ii).
If $\eps_{-k}(\mb)=0$, then we understand $n_e(\mb)=\infty$.

We define $F_{\mb,\mb'}^{-k}$
and $E_{\mb,\mb'}^{-k}$ by the coefficients of the following expansion:
\eqn
F_{-k}P_\theta(\mb)\tvac&=&\sum_{\mb'}F_{\mb,\mb'}^{-k}P_\theta(\mb')\tvac,\\
E_{-k}P_\theta(\mb)\tvac&=&\sum_{\mb'}E_{\mb,\mb'}^{-k}P_\theta(\mb')\tvac,
\eneqn
as given in Theorems~\ref{th:F-k} and \ref{th:E-k}.
Put $\ell=\varepsilon_{-k}(\mb)$ and $\ell'=\varepsilon_{-k}(\mb')$. \\
\begin{prop}
The conditions \eqref{eq:4}, \eqref{eq:6},
\eqref{eq:11.5} and \eqref{eq:15} hold,
namely, we have
\be[{\rm(a)}]
\item if  $\mb'=\tF_{-k}(\mb)$, then
$F_{\mb,\mb'}^{-k}\in q^{-\ell}(1+q\A_0)$,
\item
if $\mb'\not=\tF_{-k}(\mb)$, then
$\ord(F_{\mb,\mb'}^{-k})\ge -\ell+f(\ell+1-\ell')=-(\ell+\ell'-1)/2$,
\item
if $\mb'\not=\tF_{-k}(\mb)$
and $\ord(F_{\mb,\mb'}^{-k})=-(\ell+\ell'-1)/2$,
then the following two conditions hold:
\be[{\rm(1)}]
\item
$\xi(\tF_{-k}(\mb))>\xi(\mb')$,
\item $\ell\ge\ell'$ or $\tF_{-k}(\mb)\in B''$.
\ee
\ee
\end{prop}
\begin{proof}
We shall write $A_j$ for $A_j^{-k}(\mb)$.
Let $n_f$ be as in Definition~\ref{def:crMt} (i).

Note that $F^{-k}_{\mb,\tF_{-k}(\mb)}\not=0$.

If $F^{-k}_{\mb,\mb'}\not=0$,
we have the following four cases. 
We shall use $[n]\in q^{1-n}(1+q\A_0)$ for $n>0$.

\noindent
\textbf{Case 1.} $\mb'=\mb-\pbw{-k+2,n}+\pbw{-k,n}$ for $n>k$. 

In this case, we have
\[
F_{\mb,\mb'}^{-k}=[m_{-k,n}+1]q^{\sum_{j>n}
(m_{-k+2,j}-m_{-k,j})}\in q^{-A_n}(1+q\A_0)
\]
and
\eqn
&&\ell=\max\{A_{j}(j\ge -k+2)\},\\
&&\ell'=\max\{A_j\ (j>n),A_n+1,A_{j}+2\ (j<n)\}.
\eneqn
If $\mb'=\tF_{-k}(\mb)$, then $\ell=A_n$ and we obtain (a).
Assume $\mb'\not=\tF_{-k}(\mb)$.
Since $A_n\le\ell,\ell'-1$, we have
$\ord(F^{-k}_{\mb,\mb'})=-A_n\ge -(\ell+\ell'-1)/2$. 
Hence we obtain (b).
If $\ord(F^{-k}_{\mb,\mb'})= -(\ell+\ell'-1)/2$, then
we have $A_n=\ell=\ell'-1$. 
Since $A_j+2\le \ell'=A_n+1$ for $j<n$, we have $n_f=n$ and $\mb'=\tF_{-k}(\mb)$,
which is a contradiction.

\smallskip
\noindent
\textbf{Case 2.} $\mb'=\mb-\pbw{-k+2,k}+\pbw{-k,k}$. 

In this case we have
\[
F_{\mb,\mb'}^{-k}=[2m_{-k,k}+2]q^{\sum_{j>k}(m_{-k+2,j}-m_{-k,j})}
\in q^{-A_k-\delta(\text{$\m_{-k+2,k}$ is even})}(1+q\A_0).
\]
\bnum
\item
Assume that $m_{-k+2,k}$ is odd.
We have $F_{\mb,\mb'}^{-k}\in q^{-A_k}(1+q\A_0)$ and
\[
\ell'=\max\{A_j\ (j>k),A_k+1, A_j+2\ (j<k)\}.
\]
If $\mb'=\tF_{-k}(\mb)$, then $\ell=A_{k}$
and (a) holds.
Assume that $\mb'\not=\tF_{-k}(\mb)$.
We have $A_k\le \ell,\ell'-1$ and hence
$\ord(F_{\mb,\mb'}^{-k})=-A_k\ge-(\ell+\ell'-1)/2$.
If $\ord(F_{\mb,\mb'}^{-k})=-(\ell+\ell'-1)/2$,
then $A_k=\ell=\ell'-1$, and we have $\mb'=\tF_{-k}(\mb)$,
which is a contradiction.
\item
Assume that $m_{-k+2,k}$ is even.
Then $\mb'\not=\tF_{-k}(\mb)$,
$F_{\mb,\mb'}^{-k}\in q^{-A_k-1}(1+q\A_0)$ and
\[
\ell'=\max\{A_j\ (j>k),A_k+3,A_j+2\ (j<k)\}.
\]
We have $A_k\le \ell,\ell'-3$ and hence
$\ord(F_{\mb,\mb'}^{-k})=-A_k-1\ge-(\ell+\ell'-1)/2$.
Hence (b) holds.
Let us show (c). Assume $\mb'\not=\tF_{-k}(\mb)$, and
$\ord(F_{\mb,\mb'}^{-k})=-(\ell+\ell'-1)/2$.
Then we have $A_k=\ell=\ell'-3$.
Hence $n_f\le k$ and we have
either $\tF_{-k}(\mb)=\mb-\delta_{i\not=k}\pbw{i,k-2}+\pbw{i,k}$
with $-k+2< i\le k$ or
$\tF_{-k}(\mb)=\mb-\delta_{k\not=1}\pbw{-k+2,k-2}+\pbw{-k+2,k}$.
Hence we have $\xi(\tF_{-k}(\mb))=\pm m_{-k,k}>-m_{-k,k}-1=\xi(\mb')$.
Hence we obtain (c) (1).
\be[(1)]
\item
Assume $\tF_{-k}(\mb)=\mb-\delta_{i\not=k}\pbw{i,k-2}+\pbw{i,k}$
with $-k+2< i\le k$. Then $k\neq1$ and
$\tE_{-k}(\tF_{-k}(\mb))=\tF_{-k}(\mb)-\pbw{i,k} +\delta_{i\not=k}\pbw{i,k-2}$.
Hence $n_e(\tF_{-k}(\mb))=i-2<k$. Hence $\tF_{-k}(\mb)\in B''$.
Therefore we obtain (c) (2).
\item
Assume $\tF_{-k}(\mb)=\mb-\delta_{k\not=1}\pbw{-k+2,k-2}+\pbw{-k+2,k}$.
Then $m_{-k+2,k}(\tF_{-k}(\mb))=m_{-k+2,k}+1$ is odd.
Hence $\tF_{-k}(\mb)\in B''$.
\ee
\ee

\noindent
\textbf{Case 3.} $\mb'=\mb-
\delta_{k\not=1}\pbw{-k+2,k-2}+\pbw{-k+2,k}$. In this case, we have
\[
F_{\mb,\mb'}^{-k}=[m_{-k+2,k}+1]
q^{\sum_{j>k}(m_{-k+2,j}-m_{-k,j})+m_{-k+2,k}-2m_{-k,k}}
\in q^{-A_{k}+\delta(\text{$m_{-k+2,k}$ is odd})}(1+q\A_0).
\]
\bnum
\item
If $m_{-k+2,k}$ is odd, then $\mb'\not=\tF_{-k}(\mb)$,
$F_{\mb,\mb'}^{-k}\in q^{-A_{k}+1}(1+q\A_0)$, and
\[
\ell'=\max\{A_j\ (j>k), A_k-1,A_j+2\ (j<k)\}.
\]
We have $A_{k}\le \ell,\ell'+1$ and hence
$\ord(F_{\mb,\mb'}^{-k})=-A_k+1\ge-(\ell+\ell'-1)/2$.
If $\ord(F_{\mb,\mb'}^{-k})=-(\ell+\ell'-1)/2$,
then
$A_{k}=\ell=\ell'+1$, and $n_f=k$.
Hence we obtain (c) (2), and $\tF_{-k}(\mb)=\mb-\pbw{-k+2,k}+\pbw{-k,k}$.
Hence
$\xi(\tF_{-k}(\mb))=m_{-k,k}+1>m_{-k,k}=\xi(\mb')$. Hence we obtain (c) (1).

\item
If $m_{-k+2,k}$ is even, then $F_{\mb,\mb'}^{-k}\in q^{-A_{k}}(1+q\A_0)$
and
\[
\ell'=\max\{A_j\ (j>k),A_k+1,A_j+2\ (j<k)\}.
\]
If $\mb'=\tF_{-k}(\mb)$, then $\ell=A_k$ and (a) is satisfied.
Assume $\mb'\not=\tF_{-k}(\mb)$.
We have $A_{k}\le \ell,\ell'-1$ and hence
$\ord(F_{\mb,\mb'}^{-k})=-A_k\ge-(\ell+\ell'-1)/2$.
If $\ord(F_{\mb,\mb'}^{-k})=-(\ell+\ell'-1)/2$,
then
$A_{k}=\ell=\ell'-1$, and hence $\mb'=\tF_{-k}(\mb)$,
which is a contradiction.
\ee

\noindent
\textbf{Case 4.} $\mb'=\mb-\delta_{i\not=k}\pbw{i,k-2}+\pbw{i,k}$
for $-k+2< i\le k$.
We have
\begin{eqnarray*}
F_{\mb,\mb'}^{-k}&=
&[m_{i,k}+1]q^{\sum_{j>k}(m_{-k+2,j}-m_{-k,j})+
2m_{-k+2,k-2}-2m_{-k,k}+\sum_{-k+2<j<i}(m_{j,k-2}-m_{j,k})} \\
&\in&q^{-A_{i-2}}(1+q\A_0),
\end{eqnarray*}
and
\[
\ell'=\max\{A_j\ (j\ge k),A_j\ (j<i-2),A_{i-2}+1, A_j+2\ (i-2<j\le k-2)\}.
\]
If $\mb'=\tF_{-k}(\mb)$, then $\ell=A_{i-2}$ and
(a) holds.
Assume  $\mb'\not=\tF_{-k}(\mb)$.
Since $A_{i-2}\le\ell,\ell'-1$,
we have $\ord(F^{-k}_{\mb,\mb'})=-A_{i-2}\ge -(\ell+\ell'-1)/2$. 
Hence we obtain (b).
If $\ord(F^{-k}_{\mb,\mb'})= -(\ell+\ell'-1)/2$, then
we have $A_{i-2}=\ell=\ell'-1$. Hence $\mb'=\tF_{-k}(\mb)$,
which is a contradiction.

%
\end{proof}

\begin{prop}
Suppose $k>0$. 
The conditions \eqref{eq:5}, \eqref{eq:7},
and \eqref{eq:16} hold, namely, we have
\be[{\rm(a)}]
\item if  $\mb'=\tE_{-k}(\mb)$, then
$E_{\mb,\mb'}^{-k}\in q^{1-\ell}(1+q\A_0)$,
\item
if $\mb'\not=\tE_{-k}(\mb)$, then
$\ord(E_{\mb,\mb'}^{-k})\ge1-\ell+e(\ell-1-\ell')=-(\ell+\ell'-1)/2$,
\item
if $\mb'\not=\tE_{-k}(\mb)$, $\ell\le\ell'+1$
and $\ord(E_{\mb,\mb'}^{-k})=-(\ell+\ell'-1)/2$,
then $b\not\in B''$.
\ee
\end{prop}
\Proof
The proof is similar to the one of the above
proposition.

We shall write $A_j$ for $A_j^{-k}(\mb)$.
Let $n_e$ be as in Definition~\ref{def:crMt} (ii).

Note that $E^{-k}_{\mb,\tE_{-k}(\mb)}\not=0$ if $\tE_{-k}(\mb)\not=0$.
If $E^{-k}_{\mb,\mb'}\not=0$,
we have the following five cases. 

\noindent
\textbf{Case 1.} $\mb'=\mb-\pbw{-k,n}+\pbw{-k+2,n}$ for $n>k$. 

In this case, we have
\[
E_{\mb,\mb'}^{-k}=(1-q^2)
[m_{-k+2,n}+1]q^{1+\sum_{j\ge n}
(m_{-k+2,j}-m_{-k,j})}\in q^{1-A_n}(1+q\A_0)
\]
and
\eqn
&&\ell=\max\{A_{j}(j\ge -k+2)\},\\
&&\ell'=\max\{A_j\ (j>n),A_n-1,A_{j}-2\ (j<n)\}.
\eneqn
If $\mb'=\tE_{-k}(\mb)$, then $\ell=A_n$ and we obtain (a).
Assume $\mb'\not=\tE_{-k}(\mb)$. 
Since $A_n\le\ell,\ell'+1$,
we have $\ord(E^{-k}_{\mb,\mb'})=1-A_n\ge -(\ell+\ell'-1)/2$. 
Hence we obtain (b).
If $\ord(E^{-k}_{\mb,\mb'})= -(\ell+\ell'-1)/2$, then
we have $A_n=\ell=\ell'+1$. 
Since $A_j\le \ell'=A_n-1$ for $j>n$, we have $n_e=n$ and $\mb'=\tE_{-k}(\mb)$,
which is a contradiction.

\smallskip
\noindent
\textbf{Case 2.} $\mb'=\mb-\pbw{-k,k}+\pbw{-k+2,k}$. 

In this case we have
\eqn
E_{\mb,\mb'}^{-k}&=&(1-q^2)[m_{-k+2,k}+1]q^{1+\sum_{j>k}(m_{-k+2,j}-m_{-k,j})
+m_{-k+2,k}-2m_{-k,k}}\\
&\in& q^{1-A_k+\delta(\text{$\m_{-k+2,k}$ is odd})}(1+q\A_0).
\eneqn
\bnum
\item
Assume that $m_{-k+2,k}$ is odd. Then $\mb'\not=\tE_{-k}(\mb)$,
$E_{\mb,\mb'}^{-k}\in q^{2-A_k}(1+q\A_0)$ and
\[
\ell'=\max\{A_j\ (j>k),A_k-3, A_j-2\ (j<k)\}.
\]
We have $A_k\le \ell,\ell'+3$ and hence
$\ord(E_{\mb,\mb'}^{-k})=2-A_k\ge-(\ell+\ell'-1)/2$.
Hence (b) holds.
If $\ord(E_{\mb,\mb'}^{-k})=-(\ell+\ell'-1)/2$,
then $A_k=\ell=\ell'+3$. Hence $\ell>\ell'+1$ and (c) holds.
\item
Assume that $m_{-k+2,k}$ is even.
Then
$E_{\mb,\mb'}^{-k}\in q^{1-A_k}(1+q\A_0)$ and
\[
\ell'=\max\{A_j\ (j>k),A_k-1,A_j-2\ (j<k)\}.
\]
If $\mb'=\tE_{-k}(\mb)$, then
$\ell=A_k$, and we obtain (a).
Assume $\mb'\not=\tE_{-k}(\mb)$.
We have $A_k\le \ell,\ell'+1$ and hence
$\ord(E_{\mb,\mb'}^{-k})=1-A_k\ge-(\ell+\ell'-1)/2$.
If $\ord(E_{\mb,\mb'}^{-k})=-(\ell+\ell'-1)/2$,
then $A_k=\ell=\ell'+1$ and $n_e=k$.
Hence $\mb'=\tE_{-k}(\mb)$, which is a contradiction.
\ee

\noindent
\textbf{Case 3.} $\mb'=\mb-\pbw{-k+2,k}+\delta_{k\not=1}\pbw{-k+2,k-2}$. 
If  $k\not=1$, we have
\eqn
E_{\mb,\mb'}^{-k}&=&(1-q^2)[2(m_{-k+2,k-2}+1)]
q^{1+\sum_{j>k}(m_{-k+2,j}-m_{-k,j})+2m_{-k+2,k-2}-2m_{-k,k}}\\
&\in& q^{-A_{k}+\delta(\text{$m_{-k+2,k}$ is odd})}(1+q\A_0).
\eneqn
If $k=1$, we have
$$E_{\mb,\mb'}^{-k}=q^{\sum_{j>k}(m_{-k+2,j}-m_{-k,j})-2m_{-k,k}}
=q^{-A_{k}+\delta(\text{$m_{-k+2,k}$ is odd})}.$$
In the both cases, we have
$$E_{\mb,\mb'}^{-k}\in  
q^{-A_{k}+\delta(\text{$m_{-k+2,k}$ is odd})}(1+q\A_0).$$
\bnum
\item
If $m_{-k+2,k}$ is odd, then $E_{\mb,\mb'}^{-k}\in q^{1-A_{k}}(1+q\A_0)$
and
\[
\ell'=\max\{A_j\ (j>k),A_k-1,A_j-2\ (j<k)\}.
\]
If $\mb'=\tE_{-k}(\mb)$, then $\ell=A_k$ and (a) is satisfied.
We have $A_{k}\le \ell,\ell'+1$ and hence
$\ord(E_{\mb,\mb'}^{-k})=1-A_k\ge-(\ell+\ell'-1)/2$.
Assume $\mb'\not=\tE_{-k}(\mb)$.
If $\ord(E_{\mb,\mb'}^{-k})=-(\ell+\ell'-1)/2$,
then
$A_{k}=\ell=\ell'+1$, and $n_e=k$. Hence $\mb'=\tE_{-k}(\mb)$,
which is a contradiction.

\item
If $m_{-k+2,k}$ is even, then $\mb'\not=\tE_{-k}(\mb)$,
$E_{\mb,\mb'}^{-k}\in q^{-A_{k}}(1+q\A_0)$, and
\[
\ell'=\max\{A_j\ (j>k), A_k+1,A_j-2\ (j<k)\}.
\]
We have $A_{k}\le \ell,\ell'-1$ and hence
$\ord(E_{\mb,\mb'}^{-k})=-A_k\ge-(\ell+\ell'-1)/2$.
Hence we obtain (b).
If $\ord(E_{\mb,\mb'}^{-k})=-(\ell+\ell'-1)/2$,
then $A_{k}=\ell=\ell'-1$.
Hence $n_e(\m)\ge k$ and $m_{-k+2,k}(\mb)$ is even.
Hence $\mb\not\in B''$.
\ee

\noindent
\textbf{Case 4.} $\mb'=\mb-\pbw{i,k}+\pbw{i,k-2}$
for $-k+2< i\le k-2$.

We have
\begin{eqnarray*}
E_{\mb,\mb'}^{-k}&=
&(1-q^2)[m_{i,k-2}+1]q^{1+\sum_{j>k}(m_{-k+2,j}-m_{-k,j})+
2m_{-k+2,k-2}-2m_{-k,k}+\sum_{-k+2<j\le i}(m_{j,k-2}-m_{j,k})} \\
&\in&q^{1-A_{i-2}}(1+q\A_0),
\end{eqnarray*}
and
\[
\ell'=\max\{A_j\ (j\ge k),\;A_j\ (j<i-2),\;A_{i-2}-1,
\; A_j-2\ (i\le j\le k-2)\}.
\]
If $\mb'=\tE_{-k}(\mb)$, then $\ell=A_{i-2}$ and
(a) holds.
Assume  $\mb'\not=\tE_{-k}(\mb)$.
Since $A_{i-2}\le\ell,\ell'+1$,
we have $\ord(E^{-k}_{\mb,\mb'})=1-A_{i-2}\ge -(\ell+\ell'-1)/2$. 
Hence we obtain (b).
If $\ord(E^{-k}_{\mb,\mb'})= -(\ell+\ell'-1)/2$, then
we have $A_{i-2}=\ell=\ell'+1$. Hence $\mb'=\tE_{-k}(\mb)$,
which is a contradiction.

\noindent
\textbf{Case 5.} $k\not=1$ and $\mb'=\mb-\pbw{k}$.
In this case,
\begin{eqnarray*}
E_{\mb,\mb'}^{-k}&=&
q^{\sum_{j>k}(m_{-k+2,j}-m_{-k,j})
-2m_{-k,k}+
1-m_{k,k}+2m_{-k+2,k-2}+\sum_{-k+2<i\le k-2}(m_{i,k-2}-m_{i,k})} \\
&\in&q^{1-A_{k-2}}(1+q\A_0),
\end{eqnarray*}
and
\[
\ell'=\max\{A_j\ (j\not=k-2), A_{k-2}-1\}.
\]
If $\mb'=\tE_{-k}(\mb)$, then $\ell=A_{k-2}$ and
(a) holds.
Assume  $\mb'\not=\tE_{-k}(\mb)$.
Since $A_{k-2}\le\ell,\ell'+1$,
we have $\ord(E^{-k}_{\mb,\mb'})=1-A_{k-2}\ge -(\ell+\ell'-1)/2$. 
Hence we obtain (b).
If $\ord(E^{-k}_{\mb,\mb'})= -(\ell+\ell'-1)/2$, then
we have $A_{k-2}=\ell=\ell'+1$. Hence $\mb'=\tE_{-k}(\mb)$,
which is a contradiction.
\QED

\begin{prop}
Let $k\in I_{>0}$.
Then the conditions in {\rm Corollary~\ref{cor:crcr}}
holds for $\tE_k$, $\tF_k$ and $\eps_k$,
with the same functions $c,e,f$.
\end{prop}
Since the proof is similar to and simpler than the one of
the preceding two propositions, we omit the proof.

As a corollary we have the following result.
We write $\vac$ for the generator $\vac_0$ of $\Vt$ for short.

\begin{thm}\label{main:cr}
\bnum
\item The morphism
$$\tVt\seteq\Uf/\sum_{k\in I}\Uf(f_k-f_{-k})\to \Vt$$
is an isomorphism.
\item
$\{\Pt(\mb)\vac\}_{\mb\in\Mt}$ is a basis of
the $\K$-vector space $\Vt$.
\item
Set 
\eqn
&&\Lt\seteq\sum_{\ell\ge0,\; i_1,\ldots,i_\ell\in I}
\A_0\tF_{i_1}\cdots\tF_{i_\ell}\vac\subset\Vt,\\
&&\Bz=\set{\tF_{i_1}\cdots\tF_{i_\ell}\vac\mod q\Lt}%
{\ell\ge0, i_1,\ldots,i_\ell\in I}.
\eneqn
Then, $\Bz$ is a basis of $\Lt/q\Lt$ and
$(\Lt,\Bz)$ is a crystal basis of $\Vt$,
and the crystal structure coincides with the one of $\Mt$.
\item 
More precisely, we have
\be[{\rm(a)}]
\item
$\Lt=\soplus_{\mb\in\Mt}\A_0\Pt(\mb)\vac$,
\item
$\Bz=\set{\Pt(\mb)\vac\mod q\Lt}{\mb\in\Mt}$,
\item
for any $k\in I$ and $\mb\in\Mt$, we have
\be[{\rm (1)}]
\item
$\tF_k\Pt(\mb)\vac\equiv\Pt(\tF_k(\mb))\vac\mod q\Lt$,
\item
$\tE_k\Pt(\mb)\vac\equiv\Pt(\tE_k(\mb))\vac\mod q\Lt$, where
we understand $\Pt(0)=0$,
\item
$\tE_k^n\Pt(\mb)\vac\in q\Lt$ if and only if $n>\eps_k(\mb)$.
\ee
\ee
\enum
\end{thm}
\Proof
Let us recall that $\Pt(\mb)\vac\in \Vt$ is the image of $\tP(\mb)\in\tVt$.
By Theorem~\ref{th:F-k}, $\{\tP(\mb)\}_{\mb\in\Mt}$ generates
$\tVt$.
Let us set
$\tL=\sum_{\mb\in\Mt}\A_0\tP(\mb)\subset\tVt$.
Then Theorem~\ref{th:crcr} implies that
$\tF_k\tP(\mb)\equiv \tP(\tF_k(\mb)) \mod q\tL$ and
$\tE_k\tP(\mb)\equiv \tP(\tE_k(\mb)) \mod q\tL$.
Hence the similar results hold
for $L_0\seteq\sum_{\mb\in\Mt}\A_0\Pt(\mb)\vac\subset\Vt$
and $\Pt(\mb)\vac$.

\smallskip
Let us show that
\eqn
\text{(A)\quad $\{\Pt(\mb)\vac \mod qL_0\}_{\mb\in\Mt}$ is linearly independent
in $L_0/qL_0$,}
\eneqn
by the induction of the $\theta$-weight (see Remark \ref{rem:tweight}).
Assume that
we have a linear  relation
$\sum_{\mb\in S}a_{\mb}\Pt(\mb)\vac\equiv0 \mod qL_0$
for a finite subset $S$ 
and $a_{\mb}\in\Q\setminus\{0\}$.
We may assume that all $\mb$ in $S$ have the same $\theta$-weight.
Take $\mb_0\in S$.
If $\mb_0$ is the empty multisegment $\emptyset$,
then $S=\{\emptyset\}$ and $\Pt(\mb_0)\vac=\vac$ is non-zero, 
which is a contradiction.
Otherwise, there exists $k$ such that $\eps_k(\mb_0)>0$
by Lemma~\ref{lem:ht}.
Applying $\tE_k$, we have
$\sum_{\mb\in S}a_{\mb}\tE_k\Pt(\mb)\vac\equiv
\sum_{\mb\in S,\;\tE_k(\mb)\not=0}
a_{\mb}\Pt(\tE_k(\mb))\vac\equiv0 \mod qL_0$.
Since $\tE_k(\mb)$ ($\tE_k(\mb)\not=0$)
are mutually distinct,
we have $a_{\mb_0}=0$ by the induction hypothesis. It is a contradiction.

Thus we have proved (A). 
Hence $\{\Pt(\mb)\vac\}_{\mb\in\Mt}$ is a basis of  $\Vt$,
which implies that $\{\tP(\mb)\}_{\mb\in\Mt}$ is a basis of  $\tVt$.
Thus we obtain (i) and (ii).

Let us show (iv) (a).
Since
$\tF_{i_1}\cdots\tF_{i\ell}\vac\equiv
\Pt(\tF_{i_1}\cdots\tF_{i\ell}\emptyset)\vac \mod qL_0$,
we have $\Lt\subset L_0$ and $L_0\subset \Lt+q L_0$.
Hence Nakayama's lemma implies
$L_0=\Lt$.
The other statements are now obvious.
\QED

\section{Global basis of $\Vt$}
\subsection{Integral form of $\Vt$}
In this section, we shall prove that
$\Vt$ has a lower global basis.
In order to see this,
we shall first prove that
$\{\Pt(\mb)\vac\}_{\mb\in\Mt}$ is a basis of the $\A$-module
$\Vt_\A$.
Recall that $\A=\Q[q,q^{-1}]$, 
and $\Vt_\A=\Uf[\gl_\infty]_\A\vac$.

\Lemma
$\Vt_\A=\soplus_{\mb\in\Mt}\A \Pt(\mb)\vac$.
\enlemma
\Proof
It is clear that
$\soplus\nolimits_{\mb\in\Mt}\A \Pt(\mb)\vac$ is stable by
the actions of $F_k^{(n)}$ by Proposition \ref{prop:div}.
Hence we obtain $\Vt_\A\subset \soplus\nolimits_{\mb\in\Mt}\A \Pt(\mb)\vac$.

We shall prove $P_\theta(\mb)\vac\in\Uf[\gl_\infty]_\A\vac$. 
It is well-known that
 $\lr{i}{j}^{(m)}$ is contained in $\Uf[\gl_\infty]_\A$,
which is also seen by Proposition~\ref{prop:div} (3).
We divide $\mb$ as 
$\mb=\mb_1+\mb_2$, where
$\mb_1=\sum_{-j<i\le j}m_{ij}\lr{i}{j}$ and
$\mb_2=\sum_{k>0}m_{k}\pbw{-k,k}$.
Then $\Pt(\mb)=P(\mb_1)\Pt(\mb_2)$ and
$P(\mb_1)\in\Uf[\gl_\infty]_\A$.
Hence we may assume from the beginning
that $\mb=\sum_{0<k\le a}m_{k}\pbw{-k,k}$.
We shall show that
$\Pt(\mb)\vac\in\Vt_\A$ by the induction on $a$.

Assume $a>1$.
Set $\mb'=\sum_{0<k\le a-4}m_{k}\pbw{-k,k}$ and $v=\Pt(\mb')\vac$.
Then $\pbw{-a+2,a-2}^{\dv{m}}v\in\Vt_\A$ for any $m$ 
by the induction hypothesis.

We shall show that
$\pbw{-a,a}^{\dv{n}}\pbw{-a+2,a-2}^{\dv{m}}v$ is contained in $\Vt_\A$
by the induction on $n$.
Since $\Pt(\mb')$ commutes with
$\pbw{a}$, $\pbw{-a}$, $\pbw{-a+2,a-2}$, $\pbw{-a+2,a}$
and $\pbw{-a,a}$,
Proposition~\ref{prop:div} (2) implies
\eqn
&&\pbw{-a}^{(2n)}\pbw{-a+2,a-2}^{\dv{n+m}}v \\
&&=\sum_{i+j+2t=2n,\;j+t=u}q^{2(n+m)i+j(j-1)/2-i(t+u)}
\pbw{a}^{(i)}\pbw{-a+2,a}^{(j)}\lr{-a}{a}^{\dv{t}}
\lr{-a+2}{-2}^{\dv{n+m-u}}v,
\end{eqnarray*}
which is contained in $\Vt_\A$.
Since $\pbw{a}^{(i)}\lr{-a+2}{a}^{(j)}\lr{-a}{a}^{\dv{t}}
\lr{-a+2}{a-2}^{\dv{n+m-u}}v$
is contained in $\Vt_\A$ if $(i,j,t,u)\not=(0,0,n,n)$
by the induction hypothesis on $n$,
$\pbw{-a,a}^{\dv{n}}\pbw{-a+2,a-2}^{\dv{m}}v$ is contained in $\Vt_\A$.

If $a=1$, we similarly prove $\Pt(\mb)\vac\in\Vt_\A$ using
Proposition~\ref{prop:div} (1) instead of (2).
\QED

\subsection{Conjugate of the PBW basis}
We will prove that the bar involution is upper triangular
with respect to the PBW basis $\{P_\theta(\mb)\}_{\mb\in\Mt}$.

First we shall prove Theorem \ref{thmgl} (4).

For $a,b\in\M$ such that $a\le b$, we denote by
$\M_{[a,b]}$ (resp.\ $\M_{\le b}$)
the set of $\mb\in\M$ of the form
$\mb=\sum_{a\le i\le j\le b}m_{i,j}\pbw{i,j}$
(resp.\ $\mb=\sum_{i\le j\le b}m_{i,j}\pbw{i,j}$).
Similarly we define $(\Mt)_{\le b}$.
For a multisegment $\mb\in\M_{\le b}$,
we divide $\mb$ into $\mb=\mb_b+\mb_{<b}$, 
where $\mb_b=\sum_{i \le b}m_{i,j}\lr{i}{b}$ and 
$\mb_{<b}=\sum_{i \le j<b}m_{i,j}\lr{i}{j}$.

\Lemma
For $n\ge0$ and $a,b\in I$ such that $a\le b$, we have
$$\ol{\pbw{a,b}^{(n)}}\in
\pbw{a,b}^{(n)}+\sum_{\mb\ltcr n\pbw{a,b}}\K P(\mb).$$
\enlemma
\Proof
We shall first show
\eq
\ol{\pbw{a,b}}\in \pbw{a,b}+\sum_{a+2\le k\le b}\pbw{k,b}\Uf
\label{eq:barab}
\eneq
 by the induction on $b-a$. 
If $a=b$, it is trivial. If $a<b$, we have
\begin{eqnarray*}
\overline{\lr{a}{b}}&=&
\langle{a}\rangle \overline{\lr{a+2}{b}}
-q^{-1}\overline{\lr{a+2}{b}}\langle{a}\rangle \\
&\in&\langle{a}\rangle
\Bigl(\lr{a+2}{b}+\sum_{a+2<k\le b}\pbw{k,b}\Uf\Bigr) 
-q^{-1}\Bigl(\lr{a+2}{b}+\sum_{a+2<k\le b}\pbw{k,b}\Uf\Bigr)
\pbw{a} \\
&\subset&\lr{a}{b}+(q-q^{-1})\lr{a+2}{b}\langle{a}\rangle 
+\sum_{a+2<k\le b}
(\pbw{k,b}\pbw{a}\Uf+\pbw{k,b}\Uf).
\end{eqnarray*}
Hence we obtain \eqref{eq:barab}.
We shall show the lemma by the induction on $n$.
We may assume $n>0$ and
$$\ol{\pbw{a,b}^{n-1}}\in
\pbw{a,b}^{n-1}+\sum_{\mb\ltcr (n-1)\pbw{a,b}}\K P(\mb).$$
Hence
we have
$$\ol{\pbw{a,b}^{n}}
=\ol{\pbw{a,b}}\;\ol{\pbw{a,b}^{n-1}}\in
\pbw{a,b}^n+\sum_{a<k\le b}\pbw{k,b}\Uf+
\sum_{\mb\ltcr (n-1)\pbw{a,b}}
\K \pbw{a,b}P(\mb).$$
For $a<k\le b$ and $\mb\in\M$ such that
$\wt(\mb)=\wt(n\pbw{a,b})-\wt(\pbw{k,b})$,
we have $\mb\in\M_{[a,b]}$ and
$\mb_b=\sum_{a\le i\le b}m_{i,b}\pbw{i,b}$
with $\sum_im_{i,b}=n-1$.
In particular,
$m_{a,b}\le n-1$.
Hence $\pbw{k,b}P(\mb)
\in \K P(\mb+\pbw{k,b})$ and $\mb+\pbw{k,b}\ltcr n\pbw{a,b}$.

If $\mb\ltcr (n-1)\pbw{a,b}$, 
then $\pbw{a,b}P(\mb)\in \K P(\pbw{a,b}+\mb)$
and $\pbw{a,b}+\mb\ltcr n\pbw{a,b}$.
\QED

\Prop\label{prop:ut}
For $\mb\in\M$,
$$\ol{P(\mb)}\in P(\mb)+\sum_{\mb[n]\ltcr\mb}\K P(\mb[n]).$$
\enprop
\Proof
Put $\mb=\sum_{i \le j \le b}m_{i,j}\lr{i}{j}$ and divide 
$\mb=\mb_b+\mb_{<b}$. 
We prove the claim by the induction on $b$ and the number 
of segments in $\mb_{b}$. 
Suppose $\mb_b=m\lr{a}{b}+\mb_1$ with $m=m_{a,b}>0$, 
where $\mb_1=\sum_{a<i \le b}m_{i,b}\lr{i}{b}$.

\noindent
(i)\quad Let us first show that
\eq
&&\ol{P(\mb_b)}\in P(\mb_b)+\sum_{\mb'\ltcr \mb_b}\K P(\mb').
\label{eq:mbb}
\eneq
We have
$
\ol{P(\mb_b)}
=\ol{P(\mb_1)}\cdot \overline{\lr{a}{b}^{(m)}}$.
Since
$\ol{P(\mb_1)}\in P(\mb_1)+\sum_{\mb'_1\ltcr \mb_1}\K P(\mb'_1)$
by the induction hypothesis, and
$\overline{\lr{a}{b}^{(m)}}
\in \lr{a}{b}^{(m)}+\sum_{\mb''\ltcr m\pbw{a,b}}\K P(\mb'')$,
we have
$$\ol{P(\mb_b)}\in
P(\mb_b)+\sum_{\mb'_1\ltcr \mb_1,\;\mb'_1\in\M_{[a+2,b]}}
\K P(\mb'_1)\pbw{a,b}^{(m)}
+\sum_{\mb'_1\lecr \mb_1,\;\mb''\ltcr m\pbw{a,b}}
\K P(\mb'_1) P(\mb'').
$$
If $\mb'_1\ltcr \mb_1$ and $\mb'_1\in\M_{[a+2,b]}$,
then $P((\mb'_1)_{<b})$ and $\pbw{a,b}^{(m)}$ commute. Hence
$P(\mb'_1)\pbw{a,b}^{(m)}=P(\mb'_1+m\pbw{a,b})$
and $\mb'_1+m\pbw{a,b}\ltcr \mb_b$.

If $\mb'_1\lecr \mb_1$, $\mb'_1\in\M_{[a+2,b]}$ and 
$\mb''\ltcr m\pbw{a,b}$, then we can write
$\mb''_b=j\pbw{a,b}+\mb_2$ with $j<m$
and $\mb_2\in\M_{[a+2,b]}$.
Hence
we have
$$P(\mb'_1)P(\mb'')\in
\K P((\mb'_1)_b)P(j\pbw{a,b})P((\mb'_1)_{<b})P(\mb_2)P(\mb''_{<b}).$$
Since $(\mb'_1)_{<b}$, $\mb_2\in \M_{[a+2,b]}$
we have
$P((\mb'_1)_{<b})P(\mb_2)P(\mb''_{<b})
\in \sum_{\mb[n]_b\in\M_{[a+2,b]}}\K P(\mb[n])$.
Hence we have
$P(\mb'_1)P(\mb'')\in
\sum_{\mb[n]_b\in\M_{[a+2,b]}}\K P((\mb'_1)_b+j\pbw{a,b}+\mb[n])
$ and $(\mb'_1)_b+j\pbw{a,b}+\mb[n]\ltcr\mb_b$.
Hence we obtain \eqref{eq:mbb}.

\smallskip
\noindent
(ii)\quad By the induction hypothesis,
$\ol{P(\mb_{<b})}\in P(\mb_{<b})+\sum_{\mb''\ltcr\mb_{<b}}\K P(\mb'')$.
Since 
$\ol{P(\mb)}=\ol{P(\mb_b)}\;\ol{P(\mb_{<b})}$, 
\eqref{eq:mbb} implies that
$$\ol{P(\mb)}\in P(\mb)+
\sum_{\mb'\ltcr \mb_b,\mb''\in\M_{<b}}\K P(\mb')P(\mb'')
+\sum_{\mb''\ltcr\mb_{<b}}\K P(\mb_{b})P(\mb'').$$
For $\mb'\ltcr \mb_b$ and $\mb''\in\M_{<b}$, we have
$$P(\mb')P(\mb'')=P(\mb'_b)P(\mb'_{<b})P(\mb'')
\in \sum_{\mb[n]\in\M_{\le b},\,\mb[n]_b=\mb'_b}\K P(\mb[n])
\subset \sum_{\mb[n]\ltcr\mb}\K P(\mb[n]).$$
For $\mb''\ltcr\mb_{<b}$, we have
$P(\mb_{b})P(\mb'')=P(\mb_{b}+\mb'')$ and
$\mb_{b}+\mb''\ltcr\mb$.
Thus we obtain the desired result.
\QED

\begin{prop}\label{barpp}
For $\mb \in \Mt$, we have
\begin{eqnarray*}
\overline{P_\theta(\mb)}\vac\in P_{\theta}(\mb)\vac+
\sum_{\mb' \in \Mt,\mb'\ltcr\mb}\K P_{\theta}(\mb')\vac.
\end{eqnarray*}
\end{prop}
\begin{proof}
First note that
\eq
&&P(\mb)\vac\in\sum_{\mb[n]\in(\Mt)_{\le b}}\K \Pt(\mb[n])\vac
\quad\text{for any $b\in I_{>0}$ and $\mb\in\M_{[-b,b]}$,}
\label{eq:MMt}
\eneq
by the weight consideration.

For $\mb \in \mathcal{M}_\theta$, $P_\theta(\mb)$ and $P(\mb)$ are equal
up to a multiple of bar-invariant scalar. Thus we have
\[
\overline{P_\theta(\mb)}\in P_\theta(\mb)+\sum_{\mb' \in\M,\;\mb'\ltcr\mb}
\K P(\mb')
\]
by Proposition~\ref{prop:ut}. Hence it is enough to show that
\begin{eqnarray}
P(\mb')\vac\in
\sum_{\mb[n]\in\Mt,\;\mb[n]\ltcr\mb}\K P_\theta(\mb[n])\vac \label{bareq1}
\end{eqnarray}
for $\mb'\in\M$ such that
$\mb'\ltcr\mb$ and $\wt(\mb')=\wt(\mb)$.
Put $\mb=\sum_{i \le j \le b}m_{i,j}\pbw{i,j}$ and write $\mb=\mb_b+\mb_{<b}$.
We prove (\ref{bareq1}) by the induction on $b$. 
By the assumption on $\mb'$, we have $\mb'\in\M_{[-b,b]}$ and
$\mb'_b\lecr\mb_b$. 
Thus $\mb'_b \in \Mt$. 
Hence $\K P(\mb')\vac=\K P_\theta(\mb'_b)P(\mb'_{<b})\vac$. 

If $\mb'_b=\mb_b$, then $\mb'_{<b}<_{\text{cry}}\mb_{<b}$,
and the induction hypothesis implies
$P(\mb'_{<b})\vac\in
\sum_{\mb[n]\in\Mt,\;\mb[n]\ltcr\mb_{<b}}
\K P_\theta(\mb[n])\vac$.
Since $\Pt(\mb'_{b})\Pt(\mb[n])=\Pt(\mb'_{b}+\mb[n])$
and $\mb'_{b}+\mb[n]\ltcr\mb$, we obtain  \eqref{bareq1}.

If $\mb'_b\ltcr\mb_b$, write $\mb'=\sum_{-b\le i\le j\le b}m'_{i,j}\pbw{i,j}$.
Set $s=m_{-b.b}-m'_{-b,b}\ge0$.
Since $\wt(\mb')=\wt(\mb)$, we have
$\sum_{j<b}m'_{-b,j}=s$.
If $s=0$, 
then $\mb'_{<b}\in\M_{[-b+2,b-2]}$, and 
$P(\mb'_{<b})\vac\in\sum_{\mb[n]\in(\Mt)_{<b}}\K\Pt(\mb[n])\vac$
by \eqref{eq:MMt}.
Then \eqref{bareq1} follows from $\mb'_{b}+\mb[n]\ltcr\mb$. 

Assume $s>0$. Since $\mb'_{<b}\in\M_{[-b,b]}$,
we have
$P(\mb'_{<b})\vac\in\sum_{\mb[n]\in(\Mt)_{\le b}}\K\Pt(\mb[n])\vac$
by  \eqref{eq:MMt}.
We may assume $(1+\theta)\wt(\mb'_{<b})=(1+\theta)\wt(\mb[n])$
(see Remark~\ref{rem:tweight}).
Hence, we have $s=2m_{-b,b}(\mb[n])+\ssum_{-b<i\le b}m_{i,b}(\mb[n])$.
In particular, $m_{-b,b}(\mb[n])\le s/2$.
We have $\mb'_b+\mb[n]\in\Mt$ and
$\Pt(\mb'_b)\Pt(\mb[n])\vac
=\Pt(\mb'_b+\mb[n])\vac$.
Since $m_{-b,b}(\mb'_b+\mb[n])\le (m_{-b,b}-s)+s/2<m_{-b,b}$,
we have $\mb'_b+\mb[n]\ltcr\mb$. Hence we obtain \eqref{bareq1}.
\QED

\subsection{Existence of a global basis}
As a consequence of the preceding subsections, we obtain the following theorem.

\begin{thm}\label{th:gl}
\bnum
\item $(\Lt,\Lt^-,\Vt_\A)$ is balanced.
\item For any $\mb\in\Mt$, there exists a unique $\Gth^\lw(\mb)
\in \Lt\cap\Vt_\A$
such that
$\ol{\Gth^\lw(\mb)}=\Gth^\lw(\mb)$ and $\Gth^\lw(\mb)\equiv \Pt(\mb)\vac\;\mod q\Lt$.
\item
$\Gth^\lw(\mb)\in \Pt(\mb)\vac+\sum_{\mb[n]\ltcr\mb}q\Q[q]\Pt(\mb[n])\vac$
for any $\mb\in\Mt$.
\item
$\{\Gth^\lw(\mb)\}_{\mb\in\Mt}$ is a basis of
the $\A$-module $\Vt_\A$, the $\A_0$-module $\Lt$
and the $\K$-vector space $\Vt$.
\end{enumerate}
\end{thm}
\Proof
We have already seen that
$\ol{\Pt(\mb)\vac}=\sum_{\mb'\lecr\mb}c_{\mb,\mb'}\Pt(\mb')\vac$
for $c_{\mb,\mb'}\in \A$ with $c_{\mb,\mb}=1$.
Let us denote by $C$ the matrix $(c_{\mb,\mb'})_{\mb,\mb'\in\Mt}$.
Then $\ol{C}C=\id$ and it is well-known that there is a matrix
$A=(a_{\mb,\mb'})_{\mb,\mb'\in\Mt}$ 
such that
$\ol{A}C=A$,
$a_{\mb,\mb'}=0$ unless $\mb'\lecr\mb$,
$a_{\mb,\mb}=1$ and $a_{\mb,\mb'}\in q\Q[q]$ for $\mb'\ltcr\mb$.
Set $\Gth^\lw(\mb)=\sum_{\mb'\lecr\mb}a_{\mb,\mb'}\Pt(\mb')\vac$.
Then we have
$\ol{\Gth^\lw(\mb)}=\Gth^\lw(\mb)$ and $\Gth^\lw(\mb)\equiv \Pt(\mb)\vac\;
\mod q\Lt$.
Since $\Gth^\lw(\mb)$ is a basis of $\Vt_\A$, we obtain the desired results.
\QED

\bigskip
\noindent
{\em Errata to ``Symmetric crystals and affine Hecke algebras of type B,
Proc. Japan Acad., 82, no. 8, 2006, 131--136''}
:
\bnum
\item
In Conjecture 3.8, $\lambda=\Lambda_{p_0}+\Lambda_{p_0^{-1}}$ should be 
read as 
$\lambda=\ssum_{a\in A}\Lambda_{a}$, where
$A=I\cap\{p_0,p_0^{-1},-p_0,-p_0^{-1}\}$.
We thank S. Ariki who informed us that
the original conjecture is false.
\item
In the two diagrams of $B_\theta(\la)$ at the end of \S\,2,
$\lambda$ should be $0$.
\item
Throughout the paper, $A^{(1)}_\ell$ should be read as $A^{(1)}_{\ell-1}$.
\enum

\end{document}